\documentclass[10pt,journal,epsfig]{IEEEtran}

\usepackage{bm}

\usepackage{graphicx}
\usepackage{amssymb}
\usepackage{cite}
\usepackage{subfigure}
\usepackage{amsmath}
\usepackage{algorithm}
\usepackage{algorithmic}
\usepackage{multirow}
\usepackage{color}

\begin{document}

\title{Proximal-Free ADMM for Decentralized Composite Optimization via Graph Simplification}

\author{Bin Wang, Jun Fang,~\IEEEmembership{Member,~IEEE},
Huiping Duan and Hongbin Li,~\IEEEmembership{Senior Member,~IEEE}
\thanks{Bin Wang, Jun Fang are with the National Key Laboratory
of Science and Technology on Communications, University of
Electronic Science and Technology of China, Chengdu 611731, China,
Email: JunFang@uestc.edu.cn}
\thanks{Huiping Duan is with the School of Information and Communication Engineering,
University of Electronic Science and Technology of China, Chengdu
611731, China, Email: huipingduan@uestc.edu.cn}
\thanks{Hongbin Li is with the Department of Electrical and Computer Engineering,
Stevens Institute of Technology, Hoboken,
NJ 07030, USA, Email:  Hongbin.Li@stevens.edu}
\thanks{This work was supported in part by the National Science
Foundation of China under Grant 61522104.}}

\maketitle

\begin{abstract}
We consider the problem of decentralized composite optimization
over a symmetric connected graph, in which each node holds its own
agent-specific private convex functions, and communications are
only allowed between nodes with direct links. A variety of
algorithms have been proposed to solve such a problem in an
alternating direction method of multiplier (ADMM) framework. Many
of these algorithms, however, need to include some proximal term
in the augmented Lagrangian function such that the resulting
algorithm can be implemented in a decentralized manner. The use of
the proximal term slows down the convergence speed because it
forces the current solution to stay close to the solution obtained
in the previous iteration. To address this issue, in this paper,
we first introduce the notion of \emph{simplest bipartite graph},
which is defined as a bipartite graph that has a minimum number of
edges to keep the graph connected. A simple two-step message
passing-based procedure is proposed to find a simplest bipartite
graph associated with the original graph. We show that the
simplest bipartite graph has some interesting properties. By
utilizing these properties, an ADMM without involving any proximal
terms can be developed to perform decentralized composite
optimization over the simplest bipartite graph. Simulation results
show that our proposed method achieves a much faster convergence
speed than existing state-of-the-art decentralized algorithms.
\end{abstract}



\section{Introduction}
A network with multiple agents is called a multi-agent network,
which can be described by a graph. In a multi-agent network, each
agent is equipped with sensing, communication and computing
abilities, such that the agents are able to collaboratively
accomplish computational tasks \cite{OlfatiFax07}. Among the
various tasks that a multi-agent system can undertake, distributed
optimization is of significant importance. There are basically two
different approaches for distributed optimization, i.e.
centralized and decentralized approaches. A centralized approach
works in a center-local (master-worker) fashion
\cite{BoydParikh11}, namely, local agents (workers) are only
connected with the center agent (master). In each iteration, each
local agent sends their local data to the center agent, and the
center agent sends the processed data back to each local agent.
This operation mode brings a heavy burden on the center agent
because the center agent needs to communicate with all local
agents. Another approach is the so called decentralized methods
which has attracted much attention over the past few years. For
decentralized methods, a center agent is no longer needed. Each
agent has only access to the information of its neighboring nodes,
namely, communications are only allowed among neighboring nodes.
Such a decentralized operation involves a very low communication
cost for each agent. In addition, decentralized methods are robust
to communication disruption, node failure or malfunctioning. Due
to these attractive merits, decentralized methods have found
applications in various fields, including information processing
over sensor networks \cite{LingTian10}, aircraft or vehicle
networks \cite{RenBeard07}, cooperative spectrum sensing in
cognitive radios \cite{MengYin11}, monitoring and optimization of
smart grids \cite{GiannakisGatsis13}, distributed control of
networked robots \cite{ZhouRoumeliotis11}, distributed machine
learning \cite{ForeroCano10, FanSundaram13} and wireless
communications \cite{GiannakisLing16}.

The research on decentralized optimization originates from 1980s
\cite{TsitsiklisBertsekas86, Bertsekas83}, the time when
large-scale networks emerged. Many earlier methods, such as the
incremental subgradient methods \cite{NedicBertsekas01a,
NedicBertsekas01b} and the incremental proximal methods
\cite{Bertsekas11}, are only applicable to a special ring type
network, which is restrictive in real applications. The recently
proposed method \cite{GurbuzbalabanOzdaglar17} also follows this
line. To accommodate general networks, a distributed subgradient
method \cite{NedicOzdaglar09} and its variants \cite{RamNedic10,
Nedic11} were proposed. Though simple and easy to implement, these
algorithms are very slow due to the use of a diminishing step
size.

Many efforts have been made to develop decentralized algorithms
with fixed step sizes. Generally, these studies can be divided
into second-order methods \cite{BajovicJakovetic15,
MokhtariLing17, MokhtariShi16} and first-order methods
\cite{MateosBazerque10, ShiLing14, ChangHong14, XuZhu16,
MengFazel15, ShiLing15b, WangFang18,ShiLing15a}. Although the
second-order methods have a fast convergence rate, they usually
incur a high computational complexity since they need to compute
the inverse of the Hessian matrix of the objective function.
Besides, second-order methods require the objective function to be
twice differentiable, which may not be satisfied for many
optimization problems. Compared with second-order methods,
first-order methods are often considered more preferable due to
their simplicity and low computational complexity. Specifically,
since consensus among local variables can be formulated as a
linear constraint, decentralized first-order methods can be
naturally developed in an ADMM or an augmented Lagrangian
multiplier (ALM) framework. In \cite{MateosBazerque10},
decentralized ADMM algorithms were proposed for solving the sparse
LASSO problem, in which a node-wise consensus formulation that
enforces consensus for each pair of nodes with direct links is
used. Based on a same consensus formulation, the work
\cite{ShiLing14} addressed a general single-objective optimization
problem, and the works \cite{ChangHong14,XuZhu16} considered
decentralized composite optimization problems with a
smooth+nonsmooth structure placed on the objective functions. In
some other works, e.g.
\cite{MengFazel15,ShiLing15b,WangFang18,ShiLing15a}, a different
consensus formation was employed by resorting to the Laplacian
matrix of the underlying graph. To enable decentralized
implementation, in these works
\cite{MengFazel15,ShiLing15b,WangFang18,ShiLing15a}, an extra
proximal term expressed as
\begin{align}
\|\boldsymbol{x}-\boldsymbol{x}(k)\|_{\boldsymbol{W}}^2=(\boldsymbol{x}-\boldsymbol{x}(k))^T\boldsymbol{W}
(\boldsymbol{x}-\boldsymbol{x}(k))
\end{align}
has to be included in the augmented Lagrangian function to
eliminate the nonseparable term that demands centralized
processing. Here $\boldsymbol{x}$ denotes the optimization
variable, $\boldsymbol{x}(k)$ represents the solution obtained in
the previous iteration, and $\boldsymbol{W}$ is a carefully chosen
positive semi-definite matrix. Although facilitating decentralized
implementation, this proximal term has the disadvantage of slowing
down the convergence speed because it forces the current solution
to stay close to the solution obtained in the previous iteration.

To further improve the convergence speed of the decentralized
ADMM, in this paper, we propose a proximal-free ADMM for
decentralized composite optimization problems by converting the
underlying graph into a simplest bipartite graph. Here the
simplest bipartite graph is defined as a bipartite graph that has
a minimum number of edges to keep the graph connected. A simple
two-step message passing-based procedure is developed to find a
simplest bipartite graph associated with the original graph. By
utilizing the properties of the simplest bipartite graph, an ADMM
without involving any extra proximal terms can be developed. The
proposed algorithm exhibits a faster convergence speed than
existing state-of-the-art decentralized algorithms. Also, since
our proposed method operates over a simplest bipartite graph with
a minimum number of edges, the amount of data to be exchanged
(among nodes) is considerably reduced compared with other methods.
Such a merit makes our proposed algorithm particularly suitable
for networks subject to stringent power and communication
constraints.

The rest of this paper is organized as follows. We first introduce
the basic assumptions and the decentralized composite optimization
problem in Section \ref{sec:problem-formulation}. We then develop
a proximal-free ADMM for single-objective optimization problems in
Section \ref{sec:PF-ADMM-Single}. In Section \ref{sec:find-SBG}, a
simple two-step message passing procedure is proposed to find a
simplest bipartite graph associated with the original graph. In
Section \ref{sec:PF-ADMM-composite}, we extend our proposed method
to decentralized composite optimization problems. Simulations
results are provided in Section \ref{sec:simulation-results},
followed by concluding remarks in Section \ref{sec:conclusion}.

\section{Problem Formulation} \label{sec:problem-formulation}
\label{sec:prob} Consider a bidirectionally connected network
consisting of $l$ nodes and $m$ edges. The network is described by
a symmetric directed graph $G = \left\{ {V,E} \right\}$, where $V$
is the set of nodes and $E$ is the set of edges. At each
iteration, every node communicates with its neighboring nodes. The
communication is assumed to be synchronized. In this paper, we
consider a decentralized composite optimization problem expressed
as
\begin{align}
\mathop {\min }\limits_{\boldsymbol{x}} \sum\limits_{i = 1}^l
{{f_i}\left( \boldsymbol{x} \right) + {g_i}\left( \boldsymbol{x}
\right)} \label{opt1}
\end{align}
where $\boldsymbol{x} \in {\mathbb{R}^n}$ is the optimization
variable shared by all the objective functions, ${f_i},
{g_i}:\mathbb{R}^{n}\rightarrow \mathbb{R}\cup \{\infty\}$ are
proper, lower semicontinuous convex functions (possibly
nondifferentiable), only known to the $i$th agent. We have the
following assumption regarding the functions ${f_i}$ and ${g_i}$.

\newtheorem{assumption}{Assumption}
\begin{assumption}
The proximal mapping of $g_i,\forall i$, defined as
\begin{align}
\text{prox}_{g_i}(\boldsymbol{y})\triangleq\arg
\min_{\boldsymbol{x}} g_i(\boldsymbol{x}) +
\frac{1}{2}\|\boldsymbol{x} - \boldsymbol{y}\|_2^2
\end{align}
has a closed-form solution. For $f_i,\forall i$, either its
proximal mapping has a closed-form solution or it is a smooth
convex function with Lipschitz continuous gradient, i.e.
\begin{align}
{\left\| {\nabla {f_i}\left( \boldsymbol{x} \right) - \nabla
{f_i}\left( \boldsymbol{y} \right)} \right\|_2} \le {L_f}{\left\|
{\boldsymbol{x} - \boldsymbol{y}} \right\|_2},
 \quad \forall \boldsymbol{x},\boldsymbol{y}
\end{align}
where $L_f$ is the Lipschitz constant. To avoid triviality, it is
assumed that the proximal mapping of $f_i+g_i$ does not have a
closed-form solution.
\end{assumption}

Assume each agent $i$ holds a local copy ${\boldsymbol{x}_i} \in
\mathbb{R}^n$ of the global variable $\boldsymbol{x}$ in problem
(\ref{opt1}). Define
\begin{align}
f( {\tilde {\boldsymbol{x}}})\triangleq\sum\limits_{i = 1}^l
{{f_i}\left( {{\boldsymbol{x}_i}} \right)}\quad g( {\tilde
{\boldsymbol{x}}}) \triangleq\sum\limits_{i = 1}^l {{g_i}\left(
{{\boldsymbol{x}_i}} \right)}
\end{align}
where $\tilde {\boldsymbol{x}} \buildrel \Delta \over = \left[
\boldsymbol{x}_1^T\phantom{0}\boldsymbol{x}_2^T\phantom{0} \ldots
\phantom{0}\boldsymbol{x}_l^T\right]^T \in \mathbb{R}^{nl}$ is a
stacked column vector. We say that $\tilde {\boldsymbol{x}}$ is
consensual if ${\boldsymbol{x}_1} = {\boldsymbol{x}_2} =  \ldots =
{\boldsymbol{x}_l}$. Thus problem (\ref{opt1}) can be reformulated
as
\begin{align}
\min_{\tilde {\boldsymbol{x}}} & \quad f( {\tilde
{\boldsymbol{x}}} ) + g( {\tilde
{\boldsymbol{x}}}) \nonumber\\
\text{s.t.} & \quad \tilde {\boldsymbol{L}}\tilde {\boldsymbol{x}}
= \boldsymbol{0} \label{opt2}
\end{align}
where $\tilde{\boldsymbol{L}}\triangleq\boldsymbol{L}\otimes
{\boldsymbol{I}_n}$, $\otimes$ denotes the Kronecker product,
${\boldsymbol{I}_n} \in \mathbb{R}^{n \times n}$ is an identity
matrix, $\boldsymbol{L} \in \mathbb{R}^{l \times l}$ is the
Laplacian matrix of the graph $G$ defined as
\begin{align}
 L_{i,j} = \left\{ \begin{array}{l}
\deg(V_i), \text{if $i = j$}\\
-1, \text{$i \ne j$, node $i$ and $j$ are connected}\\
0, \text{otherwise}
\end{array} \right. \label{Laplacian-matrix}
\end{align}
in which $\deg(V_i)$ is the degree of the $i$th node. Since
$\text{null}(\boldsymbol{L})=\text{span}\{\boldsymbol{1}\}$, we
know that the variable ${\tilde {\boldsymbol{x}}}$ is consensual
if the constraint in problem (\ref{opt2}) is satisfied.

\begin{figure*}[!t]
\centering
\begin{tabular}{cc}
\includegraphics[width=4.5cm,height=4cm]{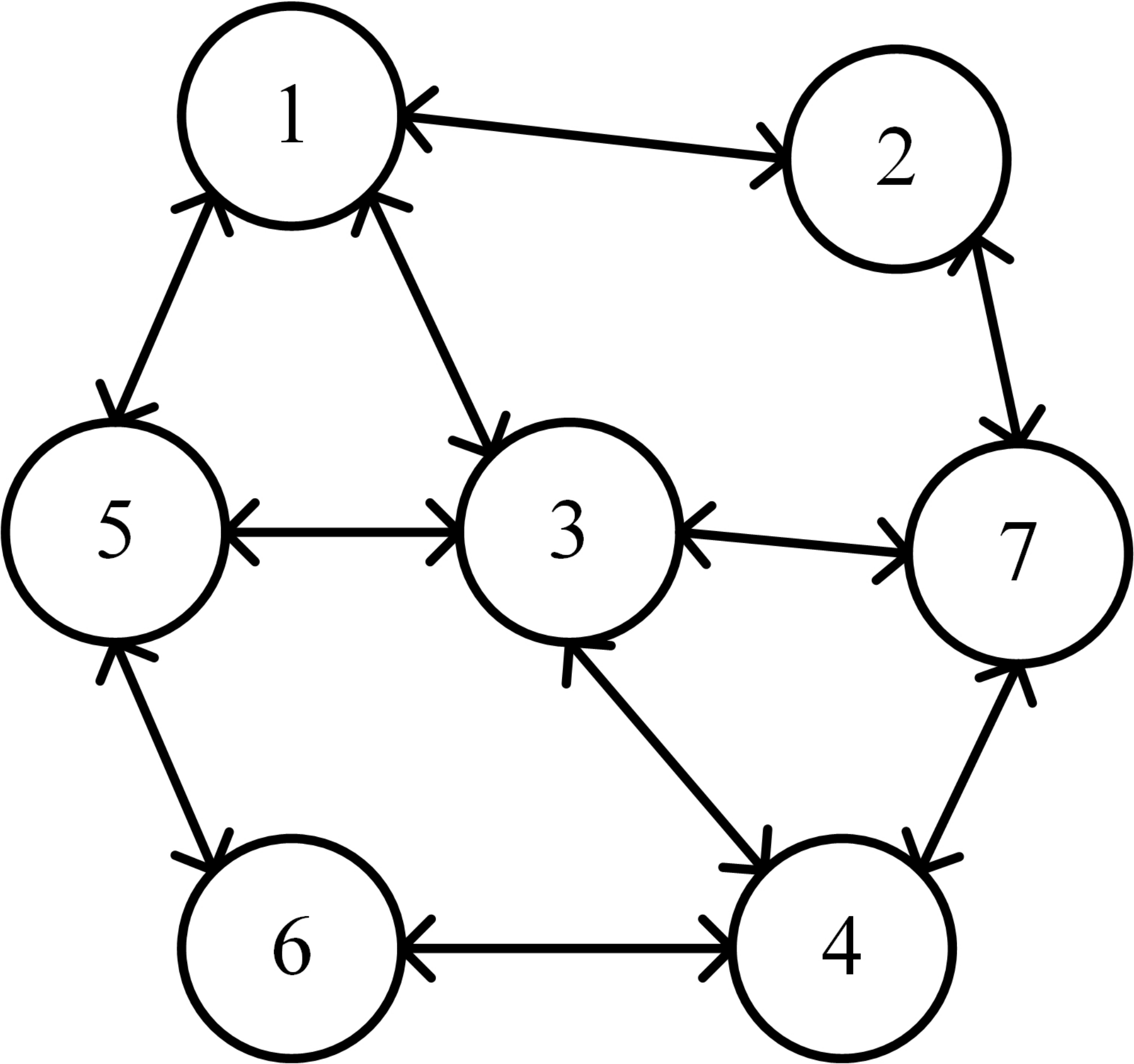}
\hspace{15pt}
\includegraphics[width=4.5cm,height=4cm]{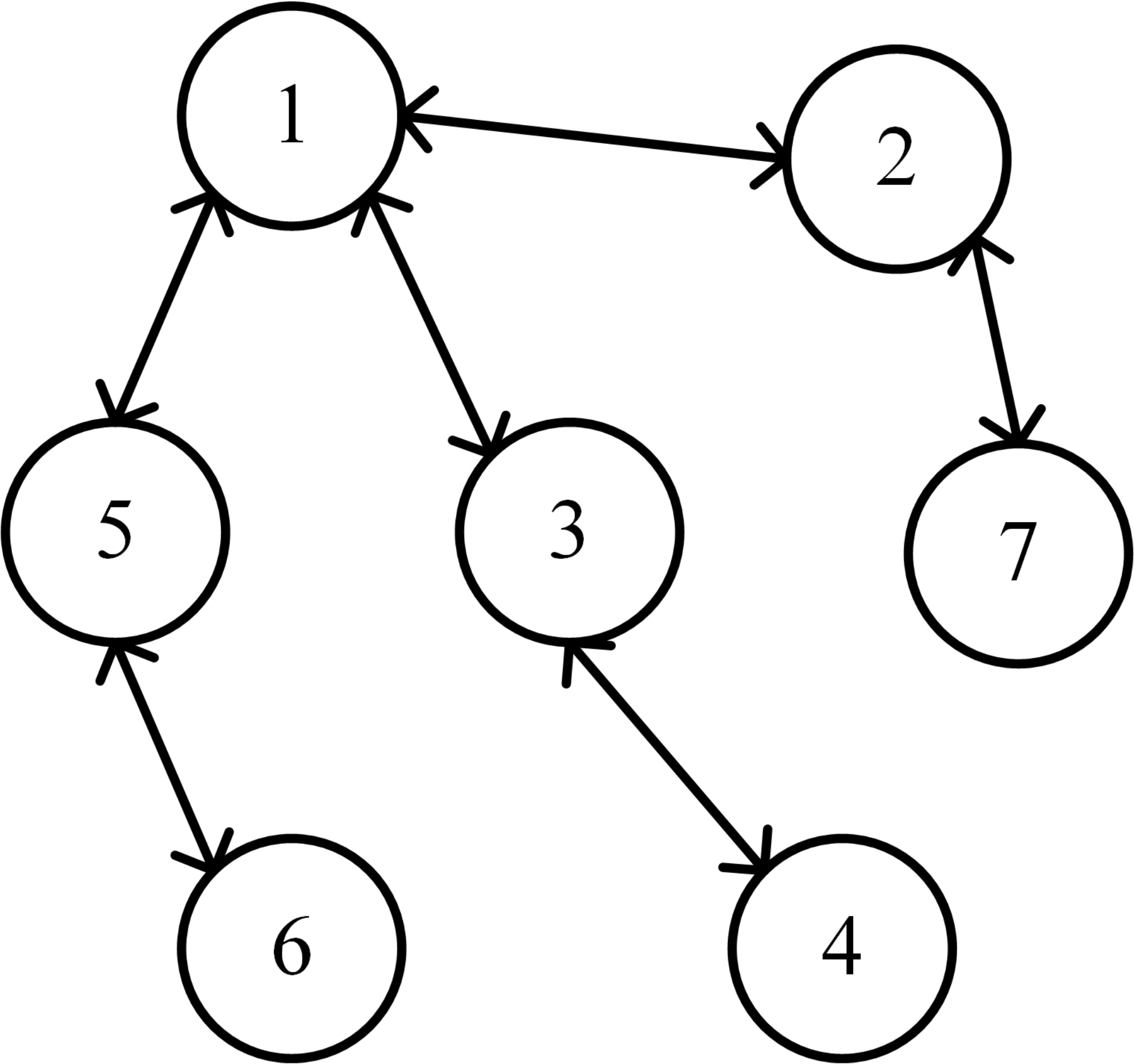}
\hspace{15pt}
\includegraphics[width=4.5cm,height=4cm]{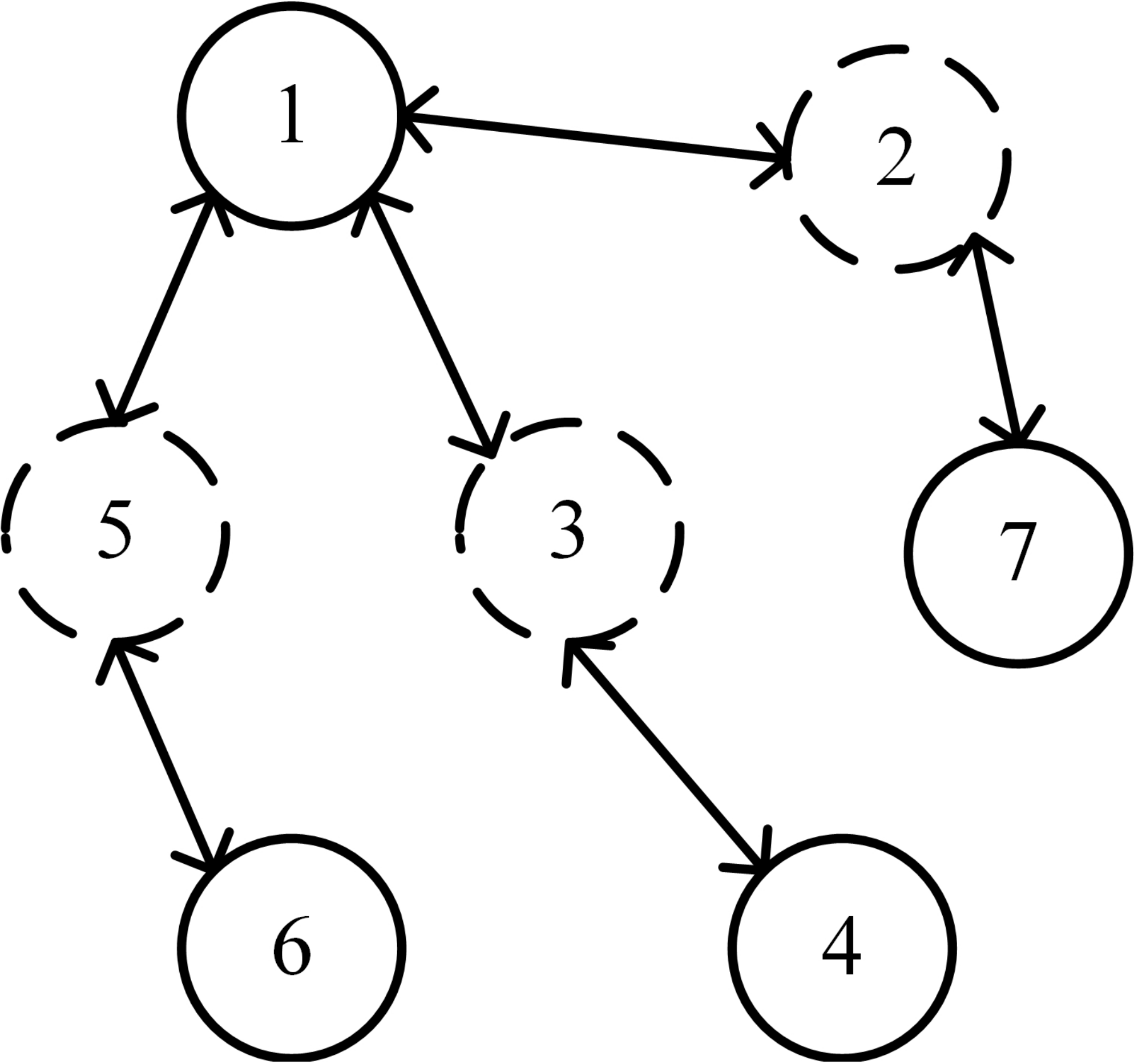}
\end{tabular}
\caption{Left: The original graph; Middle: The corresponding
minimum spanning tree; Right: The simplest bipartite graph (nodes
belonging to the set $H$ are plotted with solid line, while those
in the set $T$ are plotted with dashed line)} \label{fig:sbg}
\end{figure*}

\section{Proximal-Free ADMM: Single-Objective Optimization} \label{sec:PF-ADMM-Single}
\label{sec:basic} To more clearly explain the idea of our proposed
algorithm, we start by discussing a single-objective optimization
problem, which is a special case of the composite optimization
problem. The single-objective problem is formulated as
\begin{align}
\label{basic:1}
\min_{\tilde {\boldsymbol{x}}} & \quad f(\tilde
{\boldsymbol{x}} ) \nonumber\\
\text{s.t.} & \quad \tilde{\boldsymbol{L}}\tilde {\boldsymbol{x}} =
\boldsymbol{0}
\end{align}
Before proceeding, we first introduce the notion of simplest
bipartite graph. Here \emph{the simplest bipartite graph} is
defined as a bipartite graph that has a minimum number of edges
keeping the graph connected (see Fig. \ref{fig:sbg}). More
precisely, given a graph $G=\{V,E\}$, a simplest bipartite graph
can be denoted as $G_{sbi}=\{H,T,E_0\}$, in which $E_0\in E$ is
chosen such that $\{V,E_0\}$ is a minimum spanning tree
(MTS)\footnote{An MST is a subset of the edges of a connected
undirected graph that connects all the vertices together, without
any cycles and with the minimum possible total number of edges,
i.e. $(l-1)$ edges for a graph with $l$ nodes.} for $G$, and $H$
and $T$ are two disjoint sets ($H\cap T=\varnothing$ and $V=H\cup
T$) such that each edge in $E_0$ connects a pair of nodes not
belonging to the same set. Specifically, the simplest bipartite
graph has the same $l$ nodes as in $G =\{ {V,E}\}$, yet only keeps
$l-1$ edges. As will be shown later, the simplest bipartite graph
can be readily obtained from an MST of the graph $G$.

Given a graph $G$, suppose its simplest bipartite graph $G_{sbi}
=\{H,T,E_{0}\}$ is now available to us. Let
$\boldsymbol{L}_{G_{sbi}}$ denote the Laplacian matrix associated
with the simplest bipartite graph. Since $G_{sbi}$ is still a
connected graph, $\tilde{\boldsymbol{L}}$ in (\ref{basic:1}) can
be replaced with
$\boldsymbol{\tilde{L}}_{G_{sbi}}\triangleq\boldsymbol{L}_{G_{sbi}}\otimes
{\boldsymbol{I}_n}$, which yields
\begin{align}
\label{basic:2}
\min_{\tilde {\boldsymbol{x}}} & \quad \sum\limits_{i
\in H} f_i(\boldsymbol{x}_i )+ \sum\limits_{i
\in T} f_i(\boldsymbol{x}_i ) \nonumber\\
\text{s.t.} & \quad \tilde{\boldsymbol{L}}_{G_{sbi}} \tilde
{\boldsymbol{x}} =\boldsymbol{0}
\end{align}
In the following, we will explore some important properties about
$\boldsymbol{\tilde{L}}_{G_{sbi}}$. As will be shown, these
properties play a crucial role in developing a proximal-free
decentralized algorithm.


\newtheorem{lemma}{Lemma}
\begin{lemma}[\emph{\cite{GodsilRoyle01}}]
\label{lemma1} Given an undirected graph $G=\{V,E\}$ consisting of
$l$ nodes and $m$ edges, its corresponding Laplacian matrix
$\boldsymbol{L}$ defined in (\ref{Laplacian-matrix}) can be
decomposed as
\begin{align}
\boldsymbol{L}=\boldsymbol{X}^T\boldsymbol{X}
\end{align}
where $\boldsymbol{X}\in \mathbb{R}^{m\times l}$ is the oriented
incidence matrix of the undirected graph $G$. The oriented
incidence matrix is the incidence matrix of any orientation of the
graph, with its rows and columns indexed by the edges and vertices
of $G$, respectively. Specifically, for an edge connecting two
vertices $i$ and $j$, the row of $\boldsymbol{X}$ used to describe
this edge has two nonzero elements, with its $i$th entry and $j$th
entry equal to $1$ and $-1$ (or $-1$ and $1$) respectively. We
have the following property regarding $\boldsymbol{X}$:
\begin{itemize}
\item If two nodes, say node $i$ and $j$, in
$G$ are unconnected, then $\boldsymbol{X}_i^T \boldsymbol{X}_j=0$,
where $\boldsymbol{X}_i$ and $\boldsymbol{X}_j$ denote the $i$th
and $j$th columns of $\boldsymbol{X}$ respectively. While if node
$i$ and $j$ are connected, then $\boldsymbol{X}_i^T
\boldsymbol{X}_j=-1$. Moreover,
$\boldsymbol{X}_{i}^T\boldsymbol{X}_{i} =\text{deg}(V_i)$, where
$\text{deg}(V_i)$ is the degree of the $i$th node.
\end{itemize}
\end{lemma}
\begin{proof}
See \cite{GodsilRoyle01}.
\end{proof}

Besides, we have the following important property regarding the
incidence matrix associated with a simplest bipartite graph.

\newtheorem{lemma2}{Lemma}
\begin{lemma}
\label{lemma2} Given a simplest bipartite graph
$G_{sbi}=\{V(H,T),E_{0}\}$, the oriented incidence matrix
$\boldsymbol{X} \in \mathbb{R}^{(l-1)\times l}$ associated with
$G_{sbi}$ is a full row rank matrix.
\end{lemma}
\begin{proof}
See Appendix \ref{appA}.
\end{proof}


Let $\boldsymbol{A}$ denote the incidence matrix of the simplest
bipartite graph $G_{sbi}=\{V(H,T),E_{0}\}$. From Lemma 1, we have
$\boldsymbol{L}_{G_{sbi}}=\boldsymbol{A}^T\boldsymbol{A}$ and
${\tilde {\boldsymbol{L}}}_{G_{sbi}}= (\boldsymbol{A}\otimes
\boldsymbol{I}_n))^T(\boldsymbol{A}\otimes \boldsymbol{I}_n)$.
According to Lemma \ref{lemma2}, the incidence matrix
$\boldsymbol{A}$ is full row rank. Therefore problem
(\ref{basic:2}) can be equivalently written as
\begin{align}
\label{basic:3}
\min_{\tilde {\boldsymbol{x}}} & \quad \sum\limits_{i
\in H} f_i(\boldsymbol{x}_i )+ \sum\limits_{i
\in T} f_i(\boldsymbol{x}_i ) \nonumber\\
\text{s.t.} & \quad (\boldsymbol{A}\otimes\boldsymbol{I}_n)
\tilde {\boldsymbol{x}} =\boldsymbol{0}
\end{align}
We divide all the columns of $\boldsymbol{A}$ into two groups,
i.e. $\boldsymbol{A}_H$ and $\boldsymbol{A}_T$, in which the
indexes of the columns in $\boldsymbol{A}_H$ ($\boldsymbol{A}_T$)
are specified by the set $H$ ($T$). Similarly, let $\tilde
{\boldsymbol{x}}_H$ ($\tilde {\boldsymbol{x}}_T$) denote a stacked
column vector consisting of the local variables hold by the nodes
whose indexes belong to the set $H$ ($T$). Thus the optimization
(\ref{basic:3}) can be rewritten as
\begin{align}
\label{basic:4}
\min_{\tilde {\boldsymbol{x}}} & \quad \sum\limits_{i
\in H} f_i(\boldsymbol{x}_i )+ \sum\limits_{i
\in T} f_i(\boldsymbol{x}_i ) \nonumber\\
\text{s.t.} & \quad (\boldsymbol{A}_H\otimes\boldsymbol{I}_n)
\tilde {\boldsymbol{x}}_H+(\boldsymbol{A}_T\otimes\boldsymbol{I}_n)
\tilde {\boldsymbol{x}}_T=\boldsymbol{0}
\end{align}
Also, to facilitate our subsequent exposition, we assume the
dimension of the local variable, $n$, is equal to $1$. In such a
case, the problem (\ref{basic:4}) is simplified as
\begin{align}
\label{basic:5}
\min_{\tilde {\boldsymbol{x}}} & \quad \sum\limits_{i
\in H} f_i(\boldsymbol{x}_i )+ \sum\limits_{i
\in T} f_i(\boldsymbol{x}_i ) \nonumber\\
\text{s.t.} & \quad \boldsymbol{A}_H
\tilde {\boldsymbol{x}}_H+\boldsymbol{A}_T
\tilde {\boldsymbol{x}}_T=\boldsymbol{0}
\end{align}
Note that the following derivations can be readily extended to the
scenario where $n>1$. Also, it is noted that the formulation of
(\ref{basic:3})--(\ref{basic:5}) involves an explicit expression
of $\boldsymbol{A}$. Nevertheless, as will be shown later, our
developed algorithm only requires the knowledge of
$\boldsymbol{L}_{G_{sbi}}=\boldsymbol{A}^T\boldsymbol{A}$, which
is readily available to us.

Since the optimization (\ref{basic:5}) is a typical two-block
problem, the well-established ADMM can be used to solve this
problem, which yields the following three sub-problems:
\begin{align}
\label{Aproxi:1} {{\tilde{\boldsymbol{x}}}_H}(k + 1)=& \arg
\min_{{\tilde{\boldsymbol{x}}}_H} \sum_{i \in
H}f_i({\boldsymbol{x}}_i)
+ \nonumber \\
& \frac{\sigma }{2}\|
{{\boldsymbol{A}}_H}{{\tilde{\boldsymbol{x}}}_H} +
{{\boldsymbol{A}}_T}{{\tilde{\boldsymbol{x}}}_T}(k)
 + \frac{1}{\sigma} \boldsymbol{\lambda}(k) \|_2^2 \nonumber\\
{\tilde{\boldsymbol{x}}_T}(k + 1) =& \arg
\min_{{\tilde{\boldsymbol{x}}}_T}
\sum_{i \in T}f_i({\boldsymbol{x}}_i) \nonumber \\
&+ \frac{\sigma}{2}\| {{\boldsymbol{A}}_H}
{{\tilde{\boldsymbol{x}}}_H}(k+1)  +
{{\boldsymbol{A}}_T}{{\tilde{\boldsymbol{x}}}_T}
+ \frac{1}{\sigma}\boldsymbol{\lambda}( k ) \|_2^2 \nonumber \\
\boldsymbol{\lambda}(k + 1) =& \boldsymbol{\lambda}(k) + \sigma(
{{\boldsymbol{A}}_H}{{\tilde{\boldsymbol{x}}}_H}(k+1)
 +  {{\boldsymbol{A}}_T}{\tilde{{\boldsymbol{x}}}_T}(k+1))
\end{align}
where $\sigma$ is a parameter of user's choice. We will show that
the proposed algorithm, without resorting to proximal terms, can
be solved in a decentralized manner.

Let us first examine the $\boldsymbol{\lambda}$-subproblem.
According to the update formula of $\boldsymbol{\lambda}(k+1)$, we
have
\begin{align}
\label{Aproxi:2}
\boldsymbol{\lambda}(k + 1)& = \boldsymbol{\lambda}(k) +
\sigma( {{\boldsymbol{A}}_H}{{\tilde{\boldsymbol{x}}}_H}(k+1)
 + {{\boldsymbol{A}}_T}{{\tilde{\boldsymbol{x}}}_T}(k+1) ) \nonumber \\
&=\boldsymbol{\lambda}(0)+\sum\limits_{j = 0}^{k}
\sigma({{\boldsymbol{A}}_H}{{\tilde{\boldsymbol{x}}}_H}(j+1)
+ {{\boldsymbol{A}}_T}{{\tilde{\boldsymbol{x}}}_T}(j+1) )
\end{align}
Since $\boldsymbol{\lambda}(0)$ can be an arbitrarily initial
vector, we simply set it to be
$\boldsymbol{\lambda}(0)=\boldsymbol{0}$. In our proposed
algorithm, the $\boldsymbol{\lambda}$-update can be merged into
the update of ${\tilde{\boldsymbol{x}}_H}(k + 1)$ and
${\tilde{\boldsymbol{x}}_T}(k + 1)$. Take the
${\tilde{\boldsymbol{x}}_H}$-subproblem as an example.
Substituting (\ref{Aproxi:2}) into the
${\tilde{\boldsymbol{x}}_H}$-subproblem yields
\begin{align}
\label{Aproxi:3}
{\tilde{\boldsymbol{x}}_H}(k + 1) =& \arg \mathop {\min }
\limits_{{\tilde{\boldsymbol{x}}}_H}\sum_{i \in H} f_i({\boldsymbol{x}}_i) +
\frac{\sigma}{2}\|  {{\boldsymbol{A}}_H}{{\tilde{\boldsymbol{x}}}_H}
 +  {{\boldsymbol{A}}_T}{{\tilde{\boldsymbol{x}}}_T}(k) \nonumber \\
&+\sum\limits_{j = 0}^{k-1}({{\boldsymbol{A}_H}}\tilde{{{\boldsymbol{x}}}}_H(j+1)
 + {{\boldsymbol{A}}_T}{{\tilde{\boldsymbol{x}}}_T}(j+1) ) \|_2^2
\end{align}
Recalling Lemma \ref{lemma1}, we know that
$\boldsymbol{D}_H\triangleq\boldsymbol{A}_H^T\boldsymbol{A}_H$ and
$\boldsymbol{D}_T\triangleq\boldsymbol{A}_T^T\boldsymbol{A}_T$ are
diagonal matrices with its diagonal elements equal to the degrees
of the corresponding nodes. Based on this result, (\ref{Aproxi:3})
can be rewritten as
\begin{align}
\label{Aproxi:6} {\tilde{\boldsymbol{x}}}_H(k + 1) =&\arg \mathop
{\min } \limits_{\tilde{\boldsymbol{x}}_H} \sum_{i \in H}
f_i(\boldsymbol{x}_i ) + \frac{\sigma }{2}\|
{\boldsymbol{D}}_H^{\frac{1}{2}} {\tilde{\boldsymbol{x}}}_H +
\tilde{\boldsymbol{d}}_H \|_2^2
\end{align}
where
\begin{align}
\label{Aproxi:6I} &\tilde{\boldsymbol{d}}_H\triangleq
[\boldsymbol{d}_{i_1}^T\phantom{0}\ldots\phantom{0}
\boldsymbol{d}_{i_{|H|}}^T]^T \quad i_1,\ldots,i_{|H|}\in H \nonumber \\
&\boldsymbol{d}_i\triangleq\mu_i^{\frac{1}{2}} \sum\limits_{j =
0}^{k-1}{\boldsymbol{x}_i}(j+1)
 +\mu_i^{ -\frac{1}{2}}\boldsymbol{A}_i^T{\boldsymbol{A}_T}
\bigg({\tilde{\boldsymbol{x}}_T}(k)+ \sum\limits_{j = 0}^{k-1}
{\tilde{\boldsymbol{x}}_T}(j+1)\bigg)
\end{align}
in which $\mu_i\triangleq\text{deg}(V_i)$ is the degree of the
$i$th node. Since ${\boldsymbol{D}}_H$ in (\ref{Aproxi:6}) is a
diagonal matrix, the problem (\ref{Aproxi:6}) can be decomposed
into a number of independent tasks:
\begin{align}
\label{Aproxi:8} &{\boldsymbol{x}}_i(k + 1) =\arg \mathop {\min }
\limits_{\boldsymbol{x}_i} f_i(\boldsymbol{x}_i ) + \frac{\sigma
}{2}\| \mu_i^{\frac{1}{2}}\boldsymbol{x}_i +\boldsymbol{d}_i
\|_2^2, \ i\in H
\end{align}
Clearly, ${\boldsymbol{x}}_i(k + 1)$ is the solution of the
proximal mapping of $f_i$. If the proximal mapping of $f_i$ has a
closed-form solution, then ${\boldsymbol{x}}_i(k + 1)$ can be
easily obtained. On the other hand, if the proximal mapping of
$f_i$ dose not have a closed-form solution, but $f_i$ is
continuously differentiable and has Lipschitz continuous gradient,
then we can replace $f_i(\boldsymbol{x}_i )$ with its upper bound
\cite{BeckTeboulle09}:
\begin{align}
f_i(\boldsymbol{x}_i(k))+
\langle\boldsymbol{x}_i-\boldsymbol{x}_i(k),\nabla
f_i(\boldsymbol{x}_i )
\rangle+\frac{L_{f_i}}{2}\|\boldsymbol{x}_i-
\boldsymbol{x}_i(k)\|_2^2 \nonumber
\end{align}
where $\boldsymbol{x}_i(k)$ is the solution obtained in the
previous iteration. The resulting problem can thus be easily
solved. We see that for node $i$, solving (\ref{Aproxi:8})
requires the knowledge of $\boldsymbol{d}_i$. Nevertheless, the
calculation of $\boldsymbol{d}_i$ only involves the $i$th node
local information and information from its neighboring (i.e.
connected) nodes. To see this, recalling Lemma \ref{lemma1}, we
have $\boldsymbol{A}_i^T\boldsymbol{A}_j=0$ if node $i$ and $j$
are unconnected, while $\boldsymbol{A}_i^T\boldsymbol{A}_j=-1$ if
node $i$ and $j$ are connected. Thus the calculation of the
following term in $\boldsymbol{d}_i$:
\begin{align}
\boldsymbol{A}_i^T{\boldsymbol{A}_T}
({\tilde{\boldsymbol{x}}_T}(k)+\sum\limits_{j = 0}^{k-1}
{\tilde{\boldsymbol{x}}_T}(j+1)) \quad\forall i \in H \nonumber
\end{align}
only requires information of the nodes that are connected with the
$i$th node. Therefore the update of
${\tilde{\boldsymbol{x}}_H}(k+1)$ can be readily conducted in a
decentralized manner.

Similarly, the ${\tilde{\boldsymbol{x}}_T}(k+1)$-subproblem can
also be decomposed into a set of tasks, with each task
independently solved by each node:
\begin{align}
\label{Aproxi:9} &{\boldsymbol{x}}_i(k + 1) =\arg \mathop {\min }
\limits_{\boldsymbol{x}_i} f_i(\boldsymbol{x}_i ) + \frac{\sigma
}{2}\| \mu_i^{\frac{1}{2}}\boldsymbol{x}_i +\boldsymbol{d}_i
\|_2^2, \ i\in T
\end{align}
where
\begin{align}
\label{Aproxi:10} \boldsymbol{d}_i\triangleq&\mu_i^{\frac{1}{2}}
\sum\limits_{j =
0}^{k-1}{\boldsymbol{x}_i}(j+1) \nonumber\\
 &+\mu_i^{-\frac{1}{2}}\boldsymbol{A}_i^T{\boldsymbol{A}_H}
\bigg({\tilde{\boldsymbol{x}}_H}(k+1)+ \sum\limits_{j = 0}^{k-1}
{\tilde{\boldsymbol{x}}_H}(j+1)\bigg),\ i\in T
\end{align}
It can be easily verified that the update of
${\tilde{\boldsymbol{x}}_T}(k+1)$ also admits a decentralized
implementation. Based on the above discussion, we see that by
utilizing the properties of the simplest bipartite graph, a
proximal-free decentralized algorithm can be developed to perform
decentralized composite optimization over the simplest bipartite
graph. For clarity, the proposed algorithm is summarized in
Algorithm \ref{alg:single}.

\begin{algorithm}
\caption{Proximal-Free ADMM for solving single-objective
optimization problems} \label{alg:single}
\begin{algorithmic}
\STATE{\textbf{Inputs}: $G_{sbi}= \{H,T,E_{0}\}$ , index sets $T$
and $H$, $\sigma$. All the initial vectors are set as
$\boldsymbol{0}$}. \STATE{\textbf{while} not converge \textbf{do}}
\STATE{\textbf{for} $i \in H$, parallelly do:}
\begin{align}
{\boldsymbol{x}}_i(k + 1) =\arg \mathop {\min }
\limits_{\boldsymbol{x}_i} f_i(\boldsymbol{x}_i ) + \frac{\sigma
}{2}\| \mu_i^{\frac{1}{2}}\boldsymbol{x}_i +\boldsymbol{d}_i
\|_2^2, \ i\in H \nonumber
\end{align}
\STATE{\textbf{end for}}\\
\textbf{Information exchange}: Let node $i$, $i \in H$, send
$\boldsymbol{x}_i(k+1)$ to its neighboring nodes.
\STATE{\textbf{for} $i \in T$, parallelly do:}
\begin{align}
{\boldsymbol{x}}_i(k + 1) =\arg \mathop {\min }
\limits_{\boldsymbol{x}_i} f_i(\boldsymbol{x}_i ) + \frac{\sigma
}{2}\| \mu_i^{\frac{1}{2}}\boldsymbol{x}_i +\boldsymbol{d}_i
\|_2^2, \ i\in T \nonumber
\end{align}
\STATE{\textbf{end for}}\\
\textbf{Information exchange}: Let node $i$, $i \in T$, send
$\boldsymbol{x}_i(k+1)$ to its neighboring nodes.
\par{\textbf{end while};}
\par{\textbf{Output: $\tilde{\boldsymbol{x}}$};}
\end{algorithmic}
\end{algorithm}

\section{Finding A Simplest Bipartite Graph} \label{sec:find-SBG}
A prerequisite of our proposed proximal-free decentralized
algorithm is to obtain a simplest bipartite graph corresponding to
the original graph. This task can be accomplished via a two-step
procedure. First, we search for an MST $G_{ms}=\{V,E_0\}$ for the
original graph $G=\{V,E\}$. Then, we divide the nodes $V$ into two
disjoint sets $H$ and $T$ to form a bipartite graph
$G_{bi}=\{H,T,E_0\}$. Such a bipartite graph is a simplest
bipartite graph because it has a minimum number of edges to keep
the graph connected. We will show that both steps can be easily
performed without involving complex procedures. It should be noted
that a simplest bipartite graph not only facilitates the
development of a proximal-free decentralized algorithm, but it
also helps substantially reduce the amount of data for
transmission within the network because a simplest bipartite graph
has a minimum number of edges to keep a graph connected.

\subsection{Finding an MST}
The problem of finding an MST of a weighted, connected graph has
been extensively studied. Many sophisticated methods have been
developed to accomplish this task, including the classic Prime
algorithm \cite{Prim57} and the Kruskal algorithm
\cite{Kruskal56}. Some of these methods can be directly applied to
our problem. Nevertheless, we choose to develop a new algorithm
for finding an MST because existing methods are very complicated
as they were designed specifically for weighted graphs.

Given a graph $G=\{V,E\}$, we develop a simple message passing
scheme to construct an MST $G_{ms}=\{V,E_{0}\}$. The details of
the proposed scheme are summarized in Algorithm
\ref{alg:find-MST}. Briefly speaking, assume a random node (say,
node $1$) is activated, with all the other nodes in a sleeping
mode. This activated node then sends a probing message to its
neighboring nodes (i.e. nodes that are connected to node $1$), say
node $2$. After receiving this message, node $2$ is activated from
sleeping and sends an acknowledgement signal back to node $1$,
such that the link between node $1$ and $2$ is established. Node
$2$ then sends the probing message to its neighboring nodes
(except node $1$). Such a procedure is repeated until all nodes
are activated. During the message passing process, if a node (node
$i$) has already been activated and it receives additional probing
signals from other nodes, a denial signal will be sent back and no
links will be established. On the other hand, if a sleeping node
receives probing signals from multiple nodes simultaneously, it
randomly selects a node to respond and establishes a link between
itself and the selected node, and sends a denial signal to the
other nodes. Since only a single edge is created every time when a
node (except the first node) is activated, the newly generated
graph will have $l-1$ edges in total. Also, it is clear that the
graph is connected because each node is connected to node $1$ by
direct connection or through a multiple hop route, and thus any
two nodes are connected through a certain route. Therefore the
graph generated by our message passing scheme is an MST for the
original graph $G$. For our proposed message passing scheme, each
node needs to communicate with their neighboring nodes only once.
Also, the scheme does not need global coordination or
synchronization, thus suitable for a decentralized implementation.

\begin{algorithm}
\caption{A message passing algorithm for finding an MST}
\label{alg:find-MST}
\begin{algorithmic}
\STATE{\textbf{Inputs}: The original graph $G=\{V,E\}$. All nodes
are sleeping} \STATE{\textbf{Begin}: Randomly choose a node
(denoted as node $1$) and activate this node. Then, let node $1$
send a probing message $p$ to its neighboring nodes.}
\STATE{\textbf{while} not all nodes are activated \textbf{do}}\\

~\\

\STATE{\textbf{for} $i$=$1$:$l$};\\
\textbf{if} node $i$ is sleeping and does
not receive $p$ from other nodes, \textbf{do}\\
\quad Keep sleeping.\\
\textbf{else if} node $i$ is sleeping and received $p$ from
another
node (say node $j$), \textbf{do}\\
\quad \textcircled{\small{1}} Activate node $i$;
\textcircled{\small{2}} Send an acknowledgement signal back to
node $j$, and establish the link between node $i$ and $j$;
\textcircled{\small{3}}
Let node $i$ send $p$ to its neighbors.\\
\textbf{else if} node $i$ is sleeping and receiving $p$ from
multiple nodes, \textbf{do}\\
\quad \textcircled{\small{1}} Activate node $i$;
\textcircled{\small{2}} Randomly select a node, say node $j$, to
respond and establish the link between node $i$ and $j$;
\textcircled{\small{3}} Let node $i$ send
$p$ to its neighboring nodes.\\
\textbf{else} node $i$ is already activated, \textbf{do} \\
\quad If receiving $p$, send a denial signal to other nodes;
otherwise do nothing. \\
\textbf{end if}; \STATE{\textbf{end for};}

~\\

\par{\textbf{end while};}
\par{\textbf{Output: $G_{ms}=\{V,E_{0}\}$};}
\end{algorithmic}
\end{algorithm}

\subsection{Obtaining A Simplest Bipartite Graph}
Given an MST $G_0=\{V,E_0\}$, we now discuss how to divide the
vertices (nodes) $V$ into two disjoint sets $H$ and $T$ such that
each edge in $E_0$ connects a pair of nodes not belonging to the
same set, i.e. $\{H,T,E_0\}$ is a simplest bipartite graph. Again,
a message passing scheme can be used to achieve this goal. Start
from a randomly selected node (denoted as node $1$), and set a
label, say $H$, on node $1$, which means that this node belongs to
the set $H$. Node $1$ sends its label information $H$ to its
neighboring nodes specified by $E_0$, say node $2$. After
receiving the label information from node $1$, node $2$ sets
itself a label $T$ which is different from what is received from
node $1$. The label information of node $2$ is then sent to its
neighboring nodes (except node $1$). This procedure is repeated
until all nodes set their respective labels. Details of the scheme
are provided in Algorithm \ref{alg:obtain-sbg}. One may wonder
whether a node will receive label information from multiple nodes.
The answer is negative. In the message passing scheme, a node only
receives label information from a single node. This can be easily
shown by contradiction. Suppose that node $i$ receives label
information from two nodes, say node $j$ and $k$. This means there
is an edge between node $i$ and $j$, and also an edge between $i$
and $k$. Consequently, there will be two routes between node $i$
and node $1$, which contradicts the fact that the underlying graph
is an MST. Since each node only receives the label information
from a single node, this message passing scheme is guaranteed to
find two disjoint sets such that nodes in the same set do not have
direct links to connect them. Thus a simplest bipartite graph can
be obtained.

From the above discussions, we see that, given a symmetric,
connected graph, a simplest bipartite graph corresponding to this
graph can be easily obtained via a simple two-step procedure which
employs a message passing algorithm to first find an MST and then
uses a similar message passing technique to convert the MST into a
bipartite graph. Throughout the whole process, no global
coordination or synchronization is required. If the topology of
the network is static, such a procedure needs to be implemented
only once before the network is used to conduct any decentralized
tasks.

\begin{algorithm}
\caption{Finding a simplest bipartite graph}
\label{alg:obtain-sbg}
\begin{algorithmic}
\STATE{\textbf{Inputs}: Given an MST $G_{ms}=\{V,E_{0}\}$.}

\STATE{\textbf{Begin}: Randomly pick a node (denoted as node $1$)
and set a label (say $H$) on node $1$. Let node $1$ send its label
information to its neighboring nodes. }
\STATE{\textbf{while} not all nodes are labeled \textbf{do}}\\

~\\

\STATE{\textbf{for} $i$=$1$:$l$}; \\
\textbf{if} node $i$ is unlabeled but receives label
information from another node, \textbf{do}\\
\quad \textcircled{\small{1}} Set itself a label which is
different from the received one; \textcircled{\small{2}} Send its
label information to its neighboring nodes.\\
\textbf{else} node $i$ is already labeled, \textbf{do} \\
\quad Nothing. \\
\textbf{end if}; \STATE{\textbf{end for};}

~\\

\par{\textbf{end while};}
\par{\textbf{Output: $G_{sbi}= \{V(H,T),E_{0}\}$};}
\end{algorithmic}
\end{algorithm}

\section{Extension to Decentralized Composite Optimization
Problems} \label{sec:PF-ADMM-composite} In the previous section,
we have developed a proximal-free ADMM method for single-objective
problems. We now discuss how to extend our proposed method to
composite optimization problems. Following a similar idea (cf.
(\ref{basic:5})), we solve the decentralized composite problem
over a simplest bipartite graph $G_{sbi}= \{V(H,T),E_{0}\}$:
\begin{align}
\mathop {\min }\limits_{\tilde{\boldsymbol{x}}} & \quad
\sum\limits_{i \in
H}\left({f_i}(\boldsymbol{x}_i)+{g_i}(\boldsymbol{x}_i)\right)
+\sum\limits_{i \in
T}\left({f_i}(\boldsymbol{x}_i)+{g_i}(\boldsymbol{x}_i)\right)
\nonumber\\
\text{s.t.} & \quad \boldsymbol{A}_H {\tilde{\boldsymbol{x}}}_H
+\boldsymbol{A}_T{\tilde{\boldsymbol{x}}}_T=\boldsymbol{0}
\label{splt1}
\end{align}
Nevertheless, when applying the ADMM to the above optimization, we
need to calculate the proximal mapping of $f_i+g_i$, which usually
does not have a closed-form solution. To circumvent this
difficulty, we resort to a variable splitting technique which
introduces an auxiliary variable $\tilde{\boldsymbol{y}}^T=
[\boldsymbol{y}^T_1,\cdots,\boldsymbol{y}^T_l]^T$ to the above
problem (\ref{splt1})
\begin{align}
\mathop {\min }\limits_{\tilde{\boldsymbol{x}}} & \quad
\sum\limits_{i \in
H}\left({f_i}(\boldsymbol{x}_i)+{g_i}(\boldsymbol{y}_i)\right)
+\sum\limits_{i \in
T}\left({f_i}(\boldsymbol{x}_i)+{g_i}(\boldsymbol{y}_i)\right)
\nonumber\\
\text{s.t.} & \quad \boldsymbol{A}_H {\tilde{\boldsymbol{x}}}_H
+\boldsymbol{A}_T{\tilde{\boldsymbol{y}}}_T=\boldsymbol{0}, \,
{\tilde{\boldsymbol{x}}}_H={\tilde{\boldsymbol{y}}}_H, \,
{\tilde{\boldsymbol{x}}}_T={\tilde{\boldsymbol{y}}}_T
\label{splt2}
\end{align}
Clearly, the problem (\ref{splt2}) consists of four blocks of
variables. It was shown \cite{ChenHe16} the direct extension of
ADMM to multi-block (more than two-block) convex minimization
problems is not necessarily convergent. Nevertheless, in the
following, we will show that this four-block minimization problem
can be simplified as a conventional two-block problem. Observe
that the constraints in (\ref{splt2}) can be compactly rewritten
as
\begin{align}
\begin{bmatrix}
\boldsymbol{A}_H & \boldsymbol{0} & \boldsymbol{0} & \boldsymbol{A}_T \\
\boldsymbol{I} & -\boldsymbol{I} & \boldsymbol{0} & \boldsymbol{0}  \\
\boldsymbol{0} & \boldsymbol{0} & -\boldsymbol{I} & \boldsymbol{I}  \\
\end{bmatrix}
\begin{bmatrix}
{\tilde{\boldsymbol{x}}}_H \\
{\tilde{\boldsymbol{y}}}_H \\
{\tilde{\boldsymbol{x}}}_T \\
{\tilde{\boldsymbol{y}}}_T \\
\end{bmatrix}
=\boldsymbol{0}
\end{align}
Define two new block variables ${\tilde{\boldsymbol{z}}}_f^T
\triangleq [{\tilde{\boldsymbol{x}}}_H^T \
{\tilde{\boldsymbol{x}}}_T^T]^T$ and
${\tilde{\boldsymbol{z}}}_g^T\triangleq
[{\tilde{\boldsymbol{y}}}_H^T \ {\tilde{\boldsymbol{y}}}_T^T]^T$,
and let
\begin{align}
&\boldsymbol{C}_H\triangleq
\begin{bmatrix}
\boldsymbol{A}_H & \boldsymbol{0}\\
\boldsymbol{I} & \boldsymbol{0} \\
\boldsymbol{0} & -\boldsymbol{I}
\end{bmatrix}, \
\boldsymbol{C}_T\triangleq
\begin{bmatrix}
\boldsymbol{0} & \boldsymbol{A}_T\\
-\boldsymbol{I} & \boldsymbol{0} \\
\boldsymbol{0} & \boldsymbol{I}
\end{bmatrix}, \nonumber \\
&\tilde{f}({\tilde{\boldsymbol{z}}}_f)\triangleq\sum\limits_{i \in
H} {f_i}(\boldsymbol{x}_i)+\sum\limits_{i \in T}
{f_i}(\boldsymbol{x}_i)=\sum\limits_{i}
{f_i}(\boldsymbol{x}_i), \nonumber\\
&\tilde{g}({\tilde{\boldsymbol{z}}}_g)\triangleq\sum\limits_{i \in
H}{g_i}(\boldsymbol{y}_i) +\sum\limits_{i \in
T}{g_i}(\boldsymbol{y}_i) =\sum\limits_{i} {g_i}(\boldsymbol{y}_i)
\end{align}
Thus the problem (\ref{splt2}) can be rewritten as
\begin{align}
\mathop {\min }\limits_{{\tilde{\boldsymbol{z}}}_f, \
{\tilde{\boldsymbol{z}}}_g} & \quad
\tilde{f}({\tilde{\boldsymbol{z}}}_f)+
\tilde{g}({\tilde{\boldsymbol{z}}}_g) \nonumber\\
\text{s.t.} & \quad \boldsymbol{C}_H{\tilde{\boldsymbol{z}}}_f
+\boldsymbol{C}_T{\tilde{\boldsymbol{z}}}_g=\boldsymbol{0}
\label{splt4}
\end{align}
We see that the problem (\ref{splt2}) now has been converted into
a two-block optimization problem for which the standard ADMM
technique can be readily applied. For clarity, the proximal-free
ADMM is summarized in Algorithm \ref{alg:composite}, in which the
update of the Lagrangian multipliers has been merged into the
update of primal variables. In the following, we show that the
proposed algorithm can be easily implemented in a decentralized
way.

\begin{algorithm}
\caption{Proximal-Free ADMM} \label{alg:composite}
\begin{algorithmic}
\STATE{\textbf{Inputs}: $G_{sbi}= \{V(H,T), E_{0}\}$ , index sets
$T$ and $H$, $\sigma$. All the initial vectors are set as
$\boldsymbol{0}$}.
\STATE{\textbf{while} not converge \textbf{do}}
\begin{align}
\label{compo1} {\tilde{\boldsymbol{z}}_f}(k + 1) =& \arg \mathop
{\min } \limits_{\tilde{\boldsymbol{z}}_f}
\tilde{f}({\tilde{\boldsymbol{z}}}_f) + \frac{\sigma}{2}\|
{{\boldsymbol{C}}_H}{{\tilde{\boldsymbol{z}}}_f}
 +  {{\boldsymbol{C}}_T}{{\tilde{\boldsymbol{z}}}_g}(k) \nonumber \\
&+\sum\limits_{j =
0}^{k-1}({{\boldsymbol{C}_H}}\tilde{{{\boldsymbol{z}}}}_f(j+1)
 + {{\boldsymbol{C}}_T}{{\tilde{\boldsymbol{z}}}_g}(j+1) ) \|_2^2 \nonumber\\
{\tilde{\boldsymbol{z}}_g}(k + 1) =& \arg \mathop {\min }
\limits_{\tilde{\boldsymbol{z}}_g}
\tilde{g}({\tilde{\boldsymbol{z}}}_g) + \frac{\sigma}{2}\|
{{\boldsymbol{C}}_H}{{\tilde{\boldsymbol{z}}}_f}(k+1)
 +  {{\boldsymbol{C}}_T}{{\tilde{\boldsymbol{z}}}_g} \nonumber \\
&+\sum\limits_{j =
0}^{k-1}({{\boldsymbol{C}_H}}\tilde{{{\boldsymbol{z}}}}_f(j+1)
 + {{\boldsymbol{C}}_T}{{\tilde{\boldsymbol{z}}}_g}(j+1) ) \|_2^2
 \nonumber
\end{align}
\par{\textbf{end while};}
\par{\textbf{Output: ${\tilde{\boldsymbol{z}}}_f^T
=[{\tilde{\boldsymbol{x}}}_H^T \ {\tilde{\boldsymbol{x}}}_T^T]^T,
{\tilde{\boldsymbol{z}}}_g^T=[{\tilde{\boldsymbol{y}}}_H^T \
{\tilde{\boldsymbol{y}}}_T^T]^T$};}
\end{algorithmic}
\end{algorithm}

According to the definitions of $\boldsymbol{C}_H$ and
$\boldsymbol{C}_T$, we have
\begin{align}
\tilde{\boldsymbol{D}}_H\triangleq\boldsymbol{C}_H^T\boldsymbol{C}_H=&
\begin{bmatrix}
\boldsymbol{A}_H^T\boldsymbol{A}_H+\boldsymbol{I} & \boldsymbol{0} \\
\boldsymbol{0} & \boldsymbol{I}
\end{bmatrix}, \nonumber \\
\boldsymbol{C}_H^T\boldsymbol{C}_T=&
\begin{bmatrix}
-\boldsymbol{I} & \boldsymbol{A}_H^T\boldsymbol{A}_T \\
\boldsymbol{0} & -\boldsymbol{I}
\end{bmatrix}
\end{align}
As aforementioned, $\boldsymbol{A}_H^T\boldsymbol{A}_H$ is a
diagonal matrix. Hence $\tilde{\boldsymbol{D}}_H$ is also a
diagonal matrix. The ${\tilde{\boldsymbol{z}}_f}(k +
1)$-subproblem in Algorithm \ref{alg:composite} can therefore be
simplified as
\begin{align}
\label{compo3} {\tilde{\boldsymbol{z}}}_f(k + 1)=\arg \mathop
{\min } \limits_{{\tilde{\boldsymbol{z}}}_f}
\tilde{f}({\tilde{\boldsymbol{z}}}_f) +
\frac{\sigma}{2}\|{\tilde{\boldsymbol{D}}}_H^{\frac{1}{2}}
{\tilde{\boldsymbol{z}}}_f+{\tilde{\boldsymbol{e}}}_f \|_2^2
\end{align}
where
\begin{align}
\label{compo4} &\tilde{\boldsymbol{e}}_f^T\triangleq
[\boldsymbol{e}_{i_1}^T \ \ldots \ \boldsymbol{e}_{i_{|H|}}^T \
\boldsymbol{e}_{j_1}^T \ \ldots \ \boldsymbol{e}_{j_{|T|}}^T]^T
\end{align}
with $i_1,\ldots,i_{|H|}\in H$ and $j_1,\ldots,j_{|T|}\in T$.
Also, $\boldsymbol{e}_i,\forall i\in H$ and $\boldsymbol{e}_i,
\forall i\in T$ are respectively given as
\begin{align}
\boldsymbol{e}_i\triangleq & (\mu_i+1)^{
-\frac{1}{2}}\bigg(\boldsymbol{w}_i -({\boldsymbol{y}_i}(k)+
\sum\limits_{j = 0}^{k-1}
{\boldsymbol{y}_i}(j+1))\bigg) \nonumber \\
&+(\mu_i+1)^{\frac{1}{2}} \sum\limits_{j =
0}^{k-1}{\boldsymbol{x}_i}(j+1), \quad i \in H \nonumber\\
\boldsymbol{e}_i \triangleq &\sum\limits_{j =
0}^{k-1}{\boldsymbol{x}_i}(j+1) -({\boldsymbol{y}_i}(k)+
\sum\limits_{j = 0}^{k-1} {\boldsymbol{y}_i}(j+1)), \quad i \in T
\label{ei-definition}
\end{align}
in which
\begin{align}
\boldsymbol{w}_i\triangleq\boldsymbol{A}_i^T{\boldsymbol{A}_T}
\bigg({\tilde{\boldsymbol{y}}_T}(k)+ \sum\limits_{j = 0}^{k-1}
{\tilde{\boldsymbol{y}}_T}(j+1)\bigg)
\end{align}
Since $\tilde{\boldsymbol{D}}_H$ is a diagonal matrix, the
optimization (\ref{compo3}) can be decomposed into a set of
independent tasks, with each task separately solved by each node:
\begin{align}
\label{compo5} &\boldsymbol{x}_i=\arg \mathop {\min }
\limits_{\boldsymbol{x}_i} f_i(\boldsymbol{x}_i)
+\frac{\sigma}{2}\|(\mu_i+1)^{\frac{1}{2}}\boldsymbol{x}_i
+\boldsymbol{e}_i\|_2^2, \quad i\in H \nonumber\\
&\boldsymbol{x}_i=\arg \mathop {\min } \limits_{\boldsymbol{x}_i}
f_i(\boldsymbol{x}_i) +\frac{\sigma}{2}\|\boldsymbol{x}_i
+\boldsymbol{e}_i\|_2^2, \quad i \in T
\end{align}
Also, the calculation of $\boldsymbol{e}_i,\forall i\in H$ and
$\boldsymbol{e}_i,\forall i\in T$ can be implemented in a
decentralized manner. To see this, note that both
$\boldsymbol{x}_i(j)$ and $\boldsymbol{y}_i(j)$ are local
information readily available to the $i$th node, and the
calculation of $\boldsymbol{w}_i$ only involves information
exchange among neighboring nodes because the $j$th element of
$\boldsymbol{A}_i^T{\boldsymbol{A}_T}$ is nonzero and equals to
$-1$ only if node $i$ and node $j$ are connected. Therefore, the
${\tilde{\boldsymbol{z}}_f}(k + 1)$-subproblem admits a
decentralized implementation.

Similar to the above derivation, the ${\tilde{\boldsymbol{z}}_g}
(k + 1)$-subproblem can also be decomposed into a set of tasks,
with each task independently solved by each node:
\begin{align}
\label{compo6} &\boldsymbol{y}_i=\arg \mathop {\min }
\limits_{\boldsymbol{y}_i} g_i(\boldsymbol{y}_i)
+\frac{\sigma}{2}\|\boldsymbol{y}_i
+\boldsymbol{c}_i\|_2^2, \quad i \in H \nonumber\\
&\boldsymbol{y}_i=\arg \mathop {\min } \limits_{\boldsymbol{y}_i}
g_i(\boldsymbol{y}_i)
+\frac{\sigma}{2}\|(\mu_i+1)^{\frac{1}{2}}\boldsymbol{y}_i
+\boldsymbol{c}_i\|_2^2, \quad i \in T
\end{align}
where $\boldsymbol{c}_i,\forall i\in H$ and $\boldsymbol{c}_i,
\forall i\in T$ are respectively given as
\begin{align}
\boldsymbol{c}_i \triangleq& \sum\limits_{j =
0}^{k-1}{\boldsymbol{y}_i}(j+1)-({\boldsymbol{x}_i}(k+1)+
\sum\limits_{j = 0}^{k-1}
{\boldsymbol{x}_i}(j+1)), \quad i \in H \nonumber\\
\boldsymbol{c}_i\triangleq & (\mu_i+1)^{ -\frac{1}{2}}
\bigg(\boldsymbol{\varpi}_i-({\boldsymbol{x}_i}(k+1)+
\sum\limits_{j = 0}^{k-1}{\boldsymbol{x}_i}(j+1))\bigg)+\nonumber \\
&(\mu_i+1)^{\frac{1}{2}} \sum\limits_{j =0}^{k-1}
{\boldsymbol{y}_i}(j+1), \quad i \in T
\end{align}
in which
\begin{align}
\boldsymbol{\varpi}_i\triangleq\boldsymbol{A}_i^T{\boldsymbol{A}_H}
({\tilde{\boldsymbol{x}}_H}(k+1)+ \sum\limits_{j = 0}^{k-1}
{\tilde{\boldsymbol{x}}_H}(j+1))
\end{align}
The calculation of $\boldsymbol{c}_i,\forall i\in H$ and
$\boldsymbol{c}_i,\forall i\in T$ can be implemented in a
decentralized manner, due to similar reasons discussed for
$\boldsymbol{e}_i$ in (\ref{ei-definition}). Therefore, the
${\tilde{\boldsymbol{z}}_g}(k+1)$-subproblem admits a
decentralized implementation.


In summary, based on the above discussions, we see that the
information only needs to be exchanged among connected nodes. More
specifically, at each iteration, the calculation of
$\boldsymbol{w}_i$ requires each node, say node $j$, in the set
$T$ sends $\boldsymbol{y}_j(k)$ to its neighboring nodes, while
the calculation of $\boldsymbol{\varpi}_i$ requires each node, say
node $j$, in the set $H$ sends $\boldsymbol{x}_i(k)$ to its
neighboring nodes. To summarize, in Algorithm \ref{alg:DPF-ADMM},
we provide details of the decentralized version of the proposed
proximal-free algorithm.

\begin{algorithm}
\caption{A Decentralized Version of Proximal-Free ADMM}
\label{alg:DPF-ADMM}
\begin{algorithmic}
\STATE{\textbf{Inputs}: $G_{sbi}= \{H,T,E_{0}\}$ , index sets $T$
and $H$, $\sigma$. All the initial vectors are set as
$\boldsymbol{0}$}. \STATE{\textbf{while} not converge \textbf{do}}
\STATE{\textbf{for} $i=1:l$, parallelly do:}
\begin{align}
&\boldsymbol{x}_i=\arg \mathop {\min } \limits_{\boldsymbol{x}_i}
f_i(\boldsymbol{x}_i)
+\frac{\sigma}{2}\|(\mu_i+1)^{\frac{1}{2}}\boldsymbol{x}_i
+\boldsymbol{e}_i\|_2^2, \ i \in H \nonumber \\
&\boldsymbol{x}_i=\arg \mathop {\min } \limits_{\boldsymbol{x}_i}
f_i(\boldsymbol{x}_i) +\frac{\sigma}{2}\|\boldsymbol{x}_i
+\boldsymbol{e}_i\|_2^2, \ i \in T \nonumber
\end{align}
\STATE{\textbf{end for}}\\
\textbf{Information exchange}: Let node $i$, $i \in H$, send
$\boldsymbol{x}_i(k+1)$ to its neighbor nodes.
\STATE{\textbf{for}
$i=1:l$, parallelly do:}
\begin{align}
&\boldsymbol{y}_i=\arg \mathop {\min } \limits_{\boldsymbol{y}_i}
g_i(\boldsymbol{y}_i) +\frac{\sigma}{2}\|\boldsymbol{y}_i
+\boldsymbol{c}_i\|_2^2, \ i \in H \nonumber\\
&\boldsymbol{y}_i=\arg \mathop {\min } \limits_{\boldsymbol{y}_i}
g_t(\boldsymbol{y}_i)
+\frac{\sigma}{2}\|(\mu_i+1)^{\frac{1}{2}}\boldsymbol{y}_i
+\boldsymbol{c}_i\|_2^2, \ i \in T \nonumber
\end{align}
\STATE{\textbf{end for}}\\
\textbf{Information exchange}: Let node $i$, $i \in T$, send
$\boldsymbol{y}_i(k+1)$ to its neighbor nodes.
\par{\textbf{end while};}
\par{\textbf{Output: ${\tilde{\boldsymbol{x}}}_f^T
=[{\tilde{\boldsymbol{x}}}_H^T \ {\tilde{\boldsymbol{x}}}_T^T]^T,
{\tilde{\boldsymbol{y}}}_g^T=[{\tilde{\boldsymbol{y}}}_H^T \
{\tilde{\boldsymbol{y}}}_T^T]^T$};}
\end{algorithmic}
\end{algorithm}

\subsection{Discussions}
We now discuss the computational complexity of our proposed
proximal-free algorithm. At each iteration, each node involves
solving two subproblems, namely, the $\boldsymbol{x}_i$-subproblem
and the $\boldsymbol{y}_i$-subproblem, which are the proximal
mappings of $f_i$ and $g_i$, respectively. The calculation
involved in solving these two subproblems usually has a low
computational complexity. For many practical problems, the
nonsmooth function $g_i$ is usually a convex penalty function such
as the $\ell_1$-norm or $\ell_{\infty}$-norm whose proximal
mapping can be reduced to simple element-wise operations. On the
other hand, in many practical applications, the function $f_i$ is
a data fitting term which has a canonical form as
\begin{align}
f_i(\boldsymbol{x}_i)=\frac{1}{2}\|\boldsymbol{M}_i
\boldsymbol{x}_i-\boldsymbol{b}_i\|_2^2
\end{align}
where $\boldsymbol{M}_i \in \mathbb{R}^{m_i\times n}$. The
proximal mapping of $f_i$ is given as
\begin{align}
\text{prox}_{f_i}(\boldsymbol{x}_i)=&(\boldsymbol{M}^T_i
\boldsymbol{M}_i+\sigma(1+\mu_i)\boldsymbol{I})^{-1} \nonumber\\
&\times(\boldsymbol{M}_i^T
\boldsymbol{b}_i-\sigma(\mu_i+1)^{\frac{1}{2}}\boldsymbol{e}_i)
\end{align}
Although the proximal mapping of $f_i$ involves computing the
inverse of an $n\times n$ matrix, this matrix to be inverted
remains unaltered throughout the iterative process. Therefore the
inverse operation needs to be performed only once. Also, by
resorting to the Woodbury identity, the complexity of the matrix
inverse can be reduced to $\mathcal{O}(m_i n^2)$. Once the inverse
is calculated, $\text{prox}_{f_i}(\boldsymbol{x}_i)$ only requires
simple calculations and has a low computational complexity.

\section{Simulation Results} \label{sec:simulation-results}
\label{sec:numer} In this section, we provide simulation results
to illustrate the performance of our proposed decentralized
proximal-free ADMM algorithm (referred to as DPF-ADMM). Two
numerical examples are considered, in which the first example aims
to solve a conventional $\ell_1+\ell_2$ compressed sensing
problem, while the other example is an $\ell_1+\ell_1$ compressed
sensing problem with both $f_i$ and $g_i$ being nonsmooth
functions. To fully evaluate the performance, our experiments are
conducted over three different graphs, i.e. a line type graph with
10 nodes and 9 edges, a fully connected graph with 10 nodes and 45
edges and a randomly generated connected graph with 10 nodes and
18 edges, see Fig. \ref{fig1}. For each graph, the two-step
procedure introduced in Section \ref{sec:find-SBG} is employed to
find a simplest bipartite graph before our proposed algorithm is
executed to solve the decentralized problems.

\begin{figure*}
 \centering
\begin{tabular}{cc}
\includegraphics[width=5.5cm,height=3.5cm]{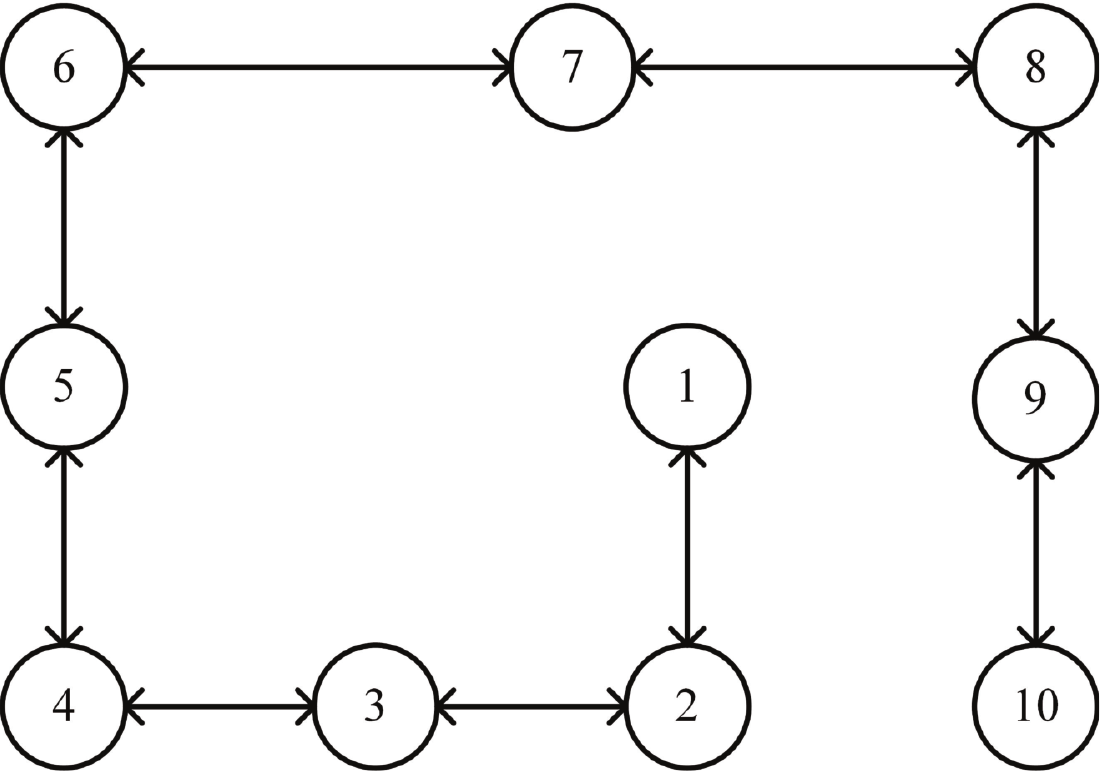}
\hspace*{5pt}
\includegraphics[width=5cm,height=4cm]{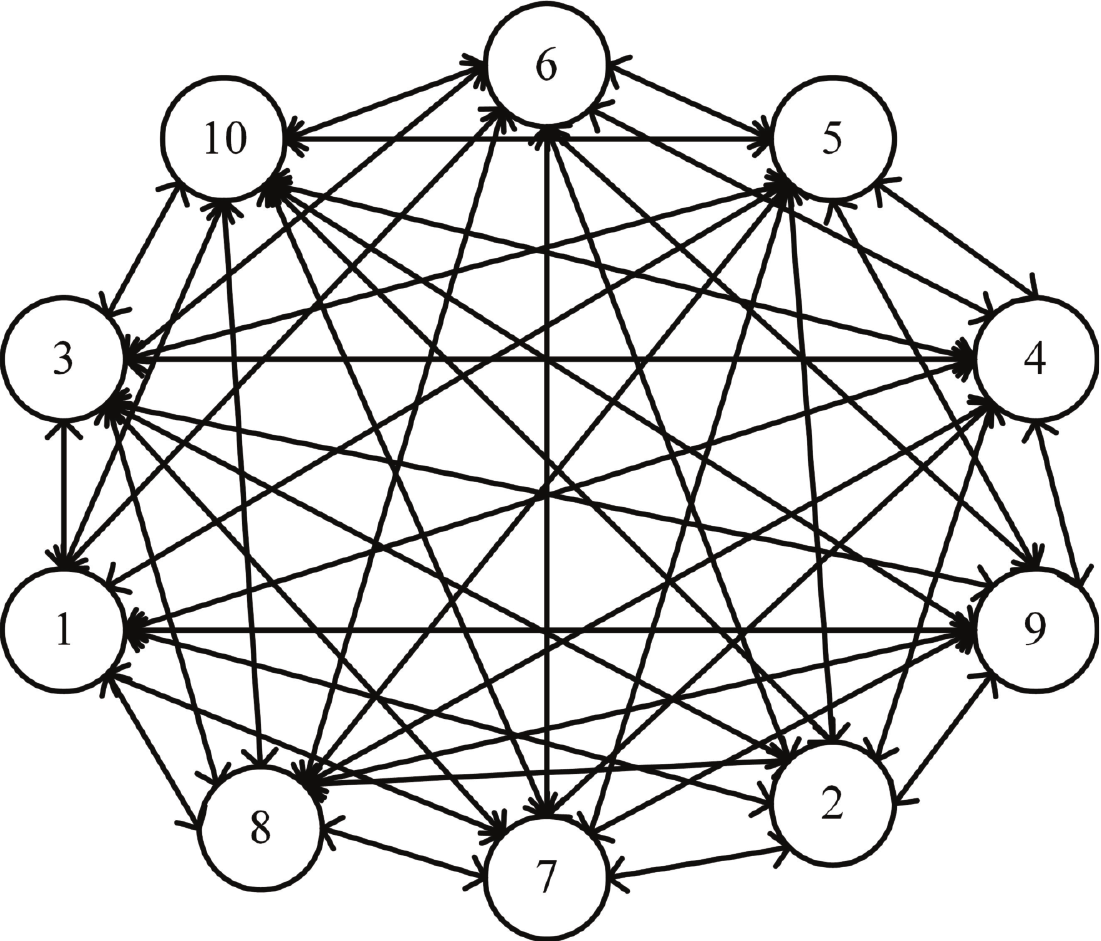}
\hspace*{5pt}
\includegraphics[width=5cm,height=4cm]{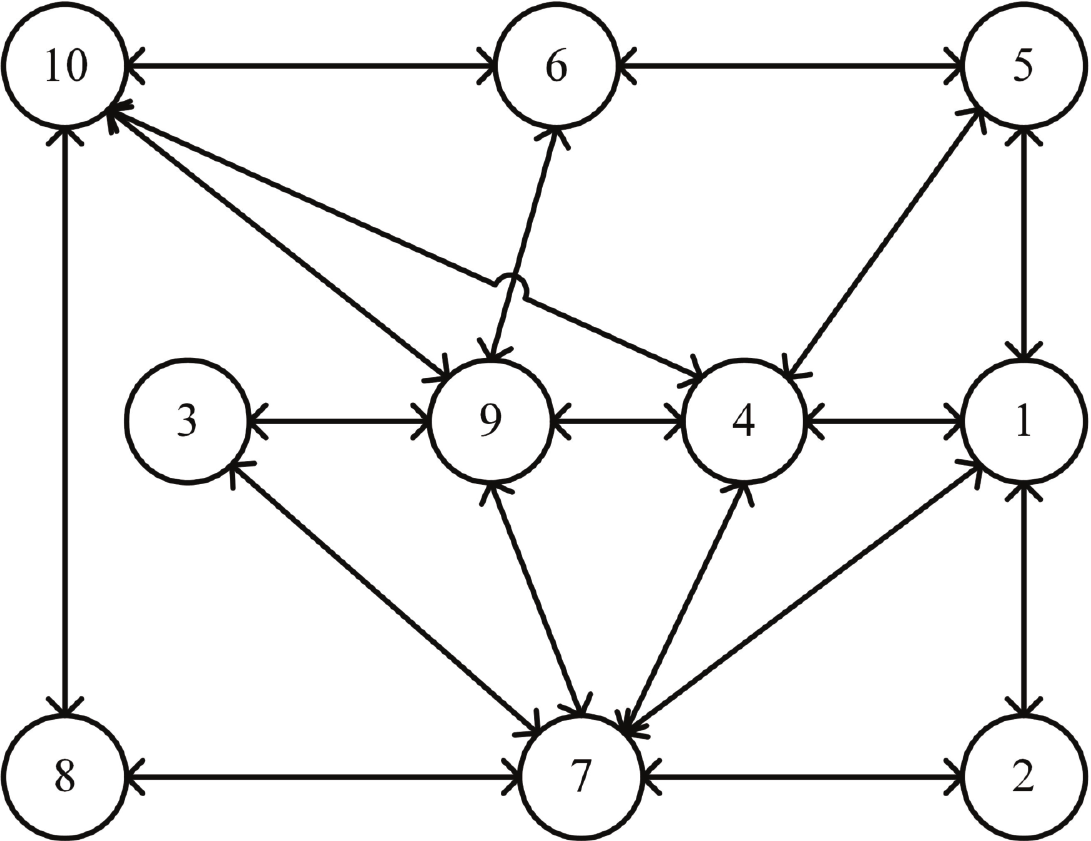}
\end{tabular}\\
\vspace*{2ex}
\begin{tabular}{cc}
\includegraphics[width=5.5cm,height=3.5cm]{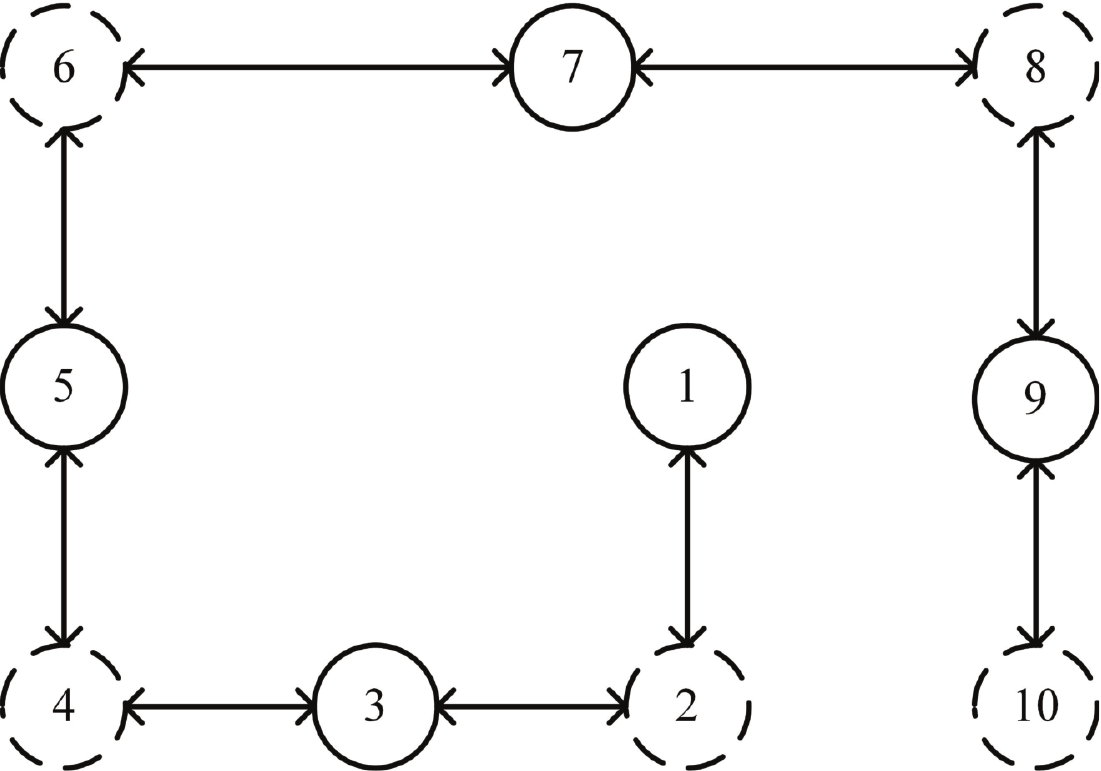}
\hspace*{5pt}
\includegraphics[width=5cm,height=4cm]{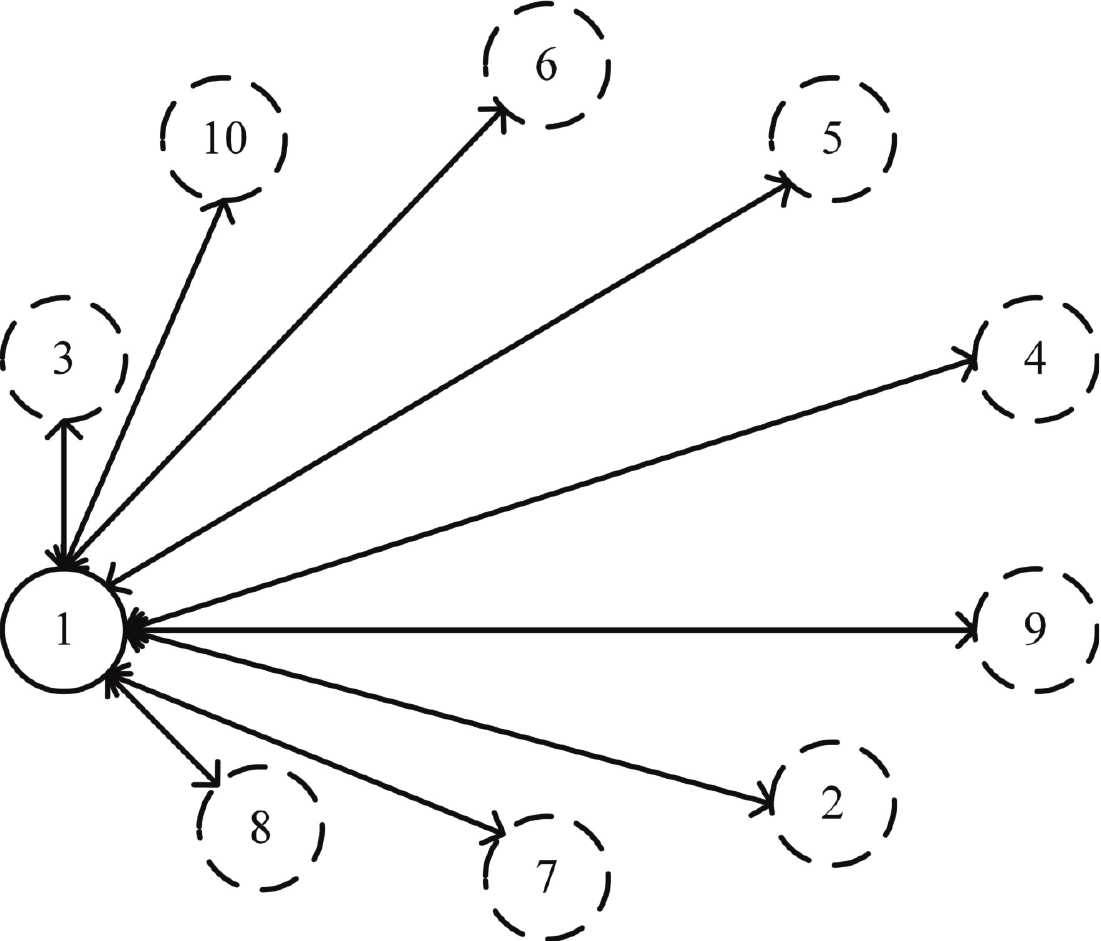}
\hspace*{5pt}
\includegraphics[width=5cm,height=4cm]{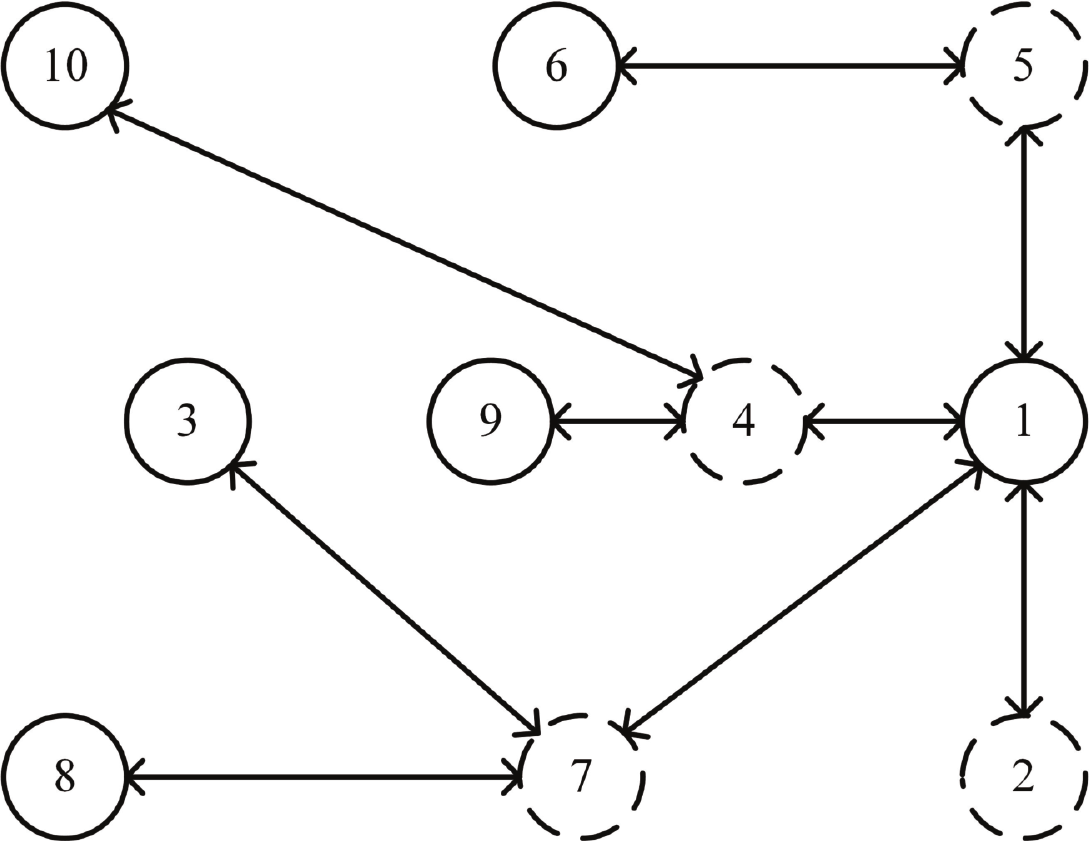}
\end{tabular}
  \caption{The original graphs and their corresponding simplest
  bipartite graphs. Top: original graphs (Left: line graph;
  Middle: fully connected graph; Right: random graph); Bottom:
  obtained simplest bipartite graphs
  (Left: The simplest bipartite graph of the line graph;
  Middle: The simplest bipartite graph of the fully connected
  graph; Right: The simplest bipartite graph of the random graph).
  The circles with solid lines denote nodes belonging to the index set $H$,
  and those with dashed lines denote nodes belonging to $T$.
  }
  \label{fig1}
\end{figure*}

\begin{figure*}[t]
    \centering
    \subfigure[Random graph.]{\includegraphics[width=2.3in]{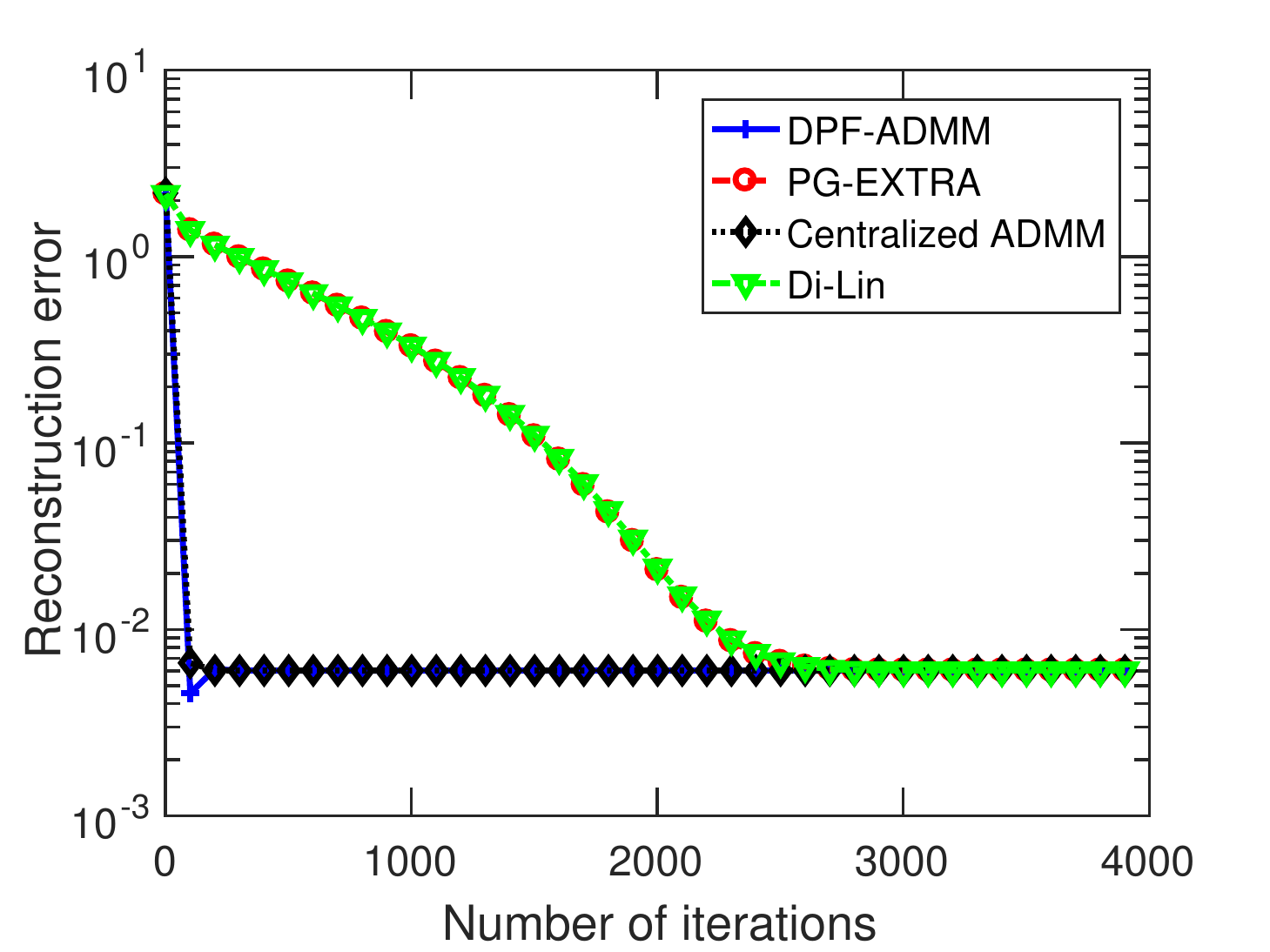}} \hfil
    \subfigure[Line graph.]{\includegraphics[width=2.3in]{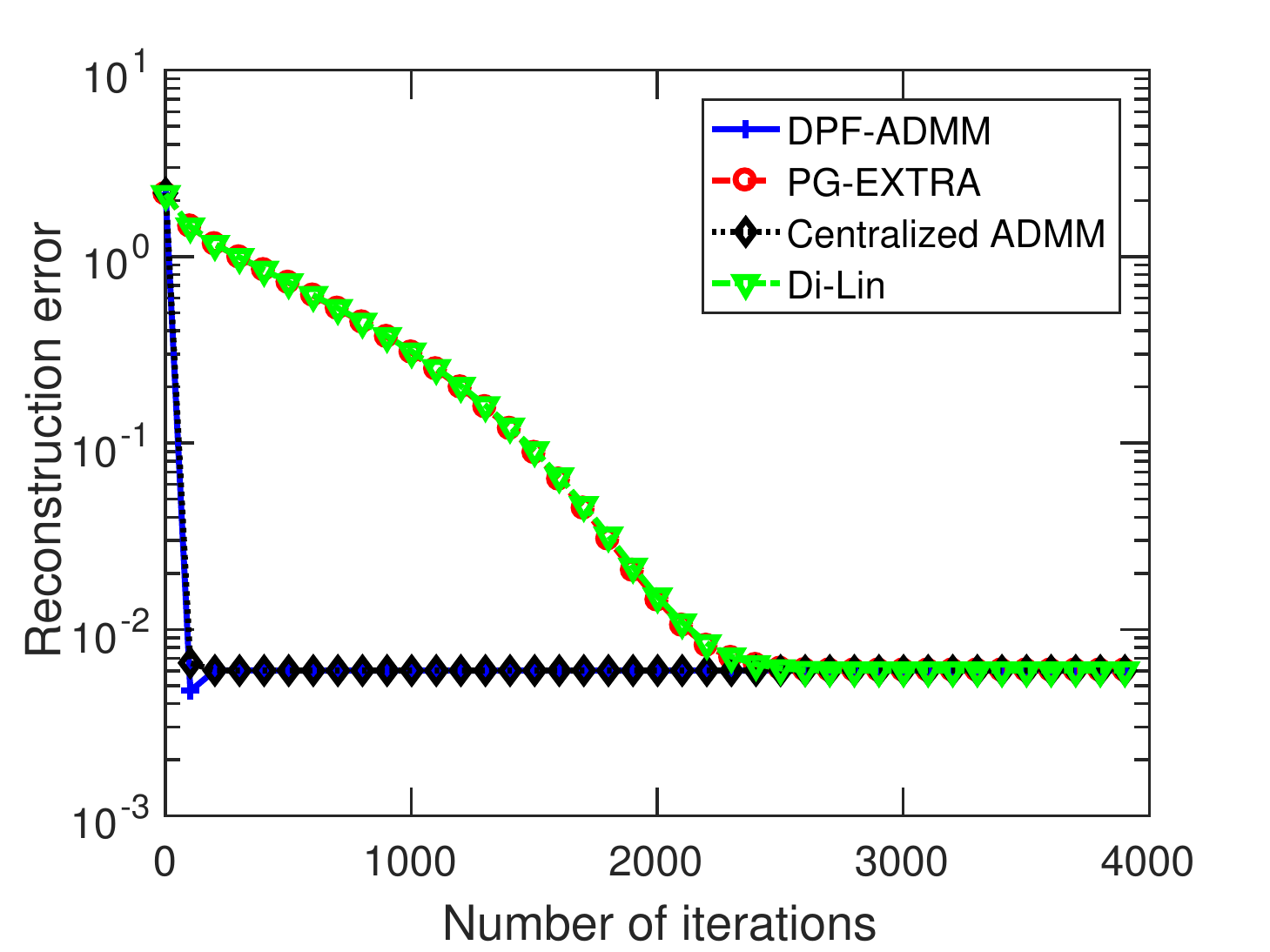}} \hfil
    \subfigure[Fully connected graph.]{\includegraphics[width=2.3in]{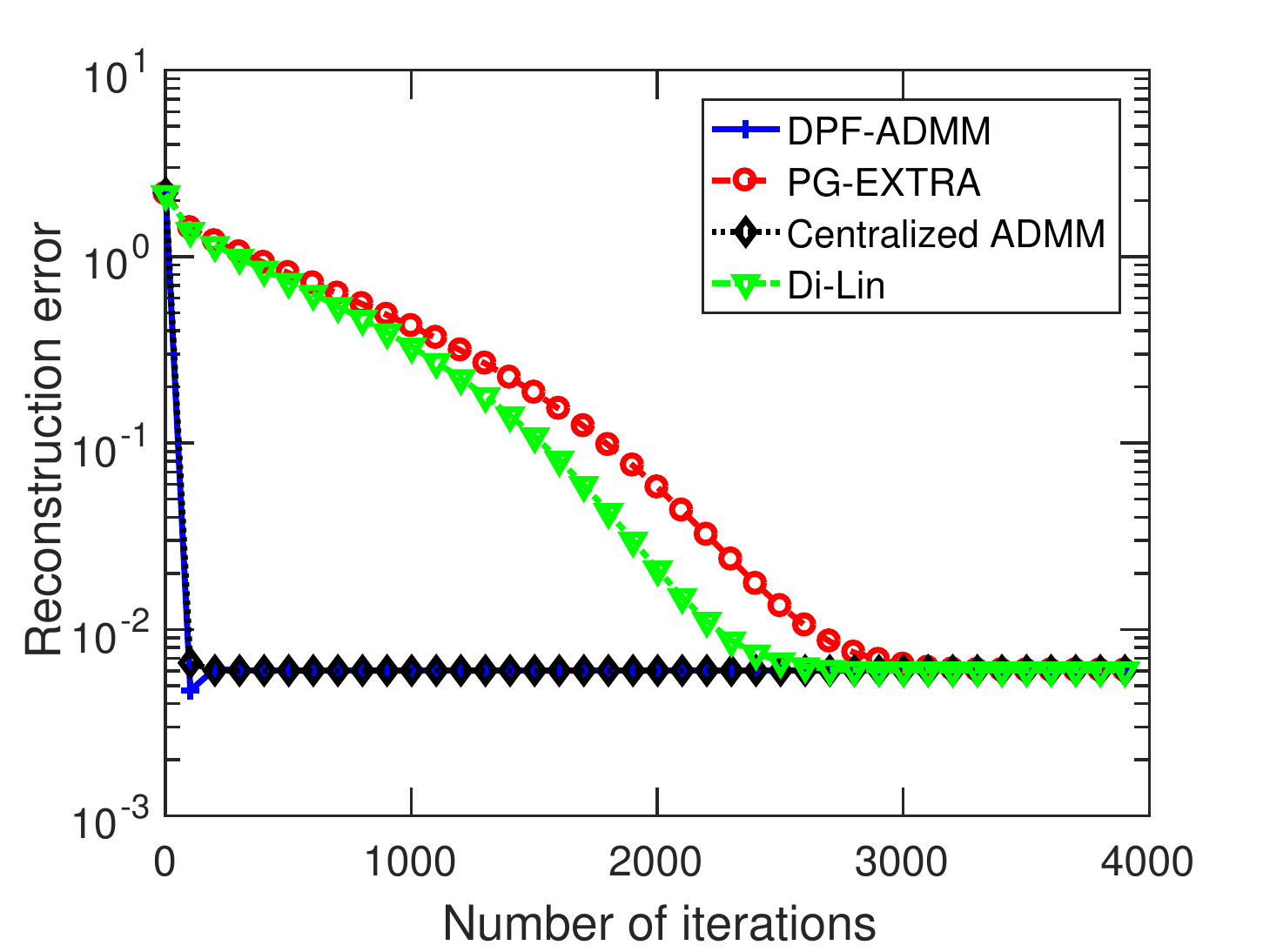}}
    \caption{
    $\ell_2$+$\ell_1$ compressed sensing: reconstruction errors vs. the number of iterations for respective algorithms.}
    \label{fig2}
\end{figure*}

\begin{figure*}[t]
    \centering
    \subfigure[Random graph.]{\includegraphics[width=2.3in]{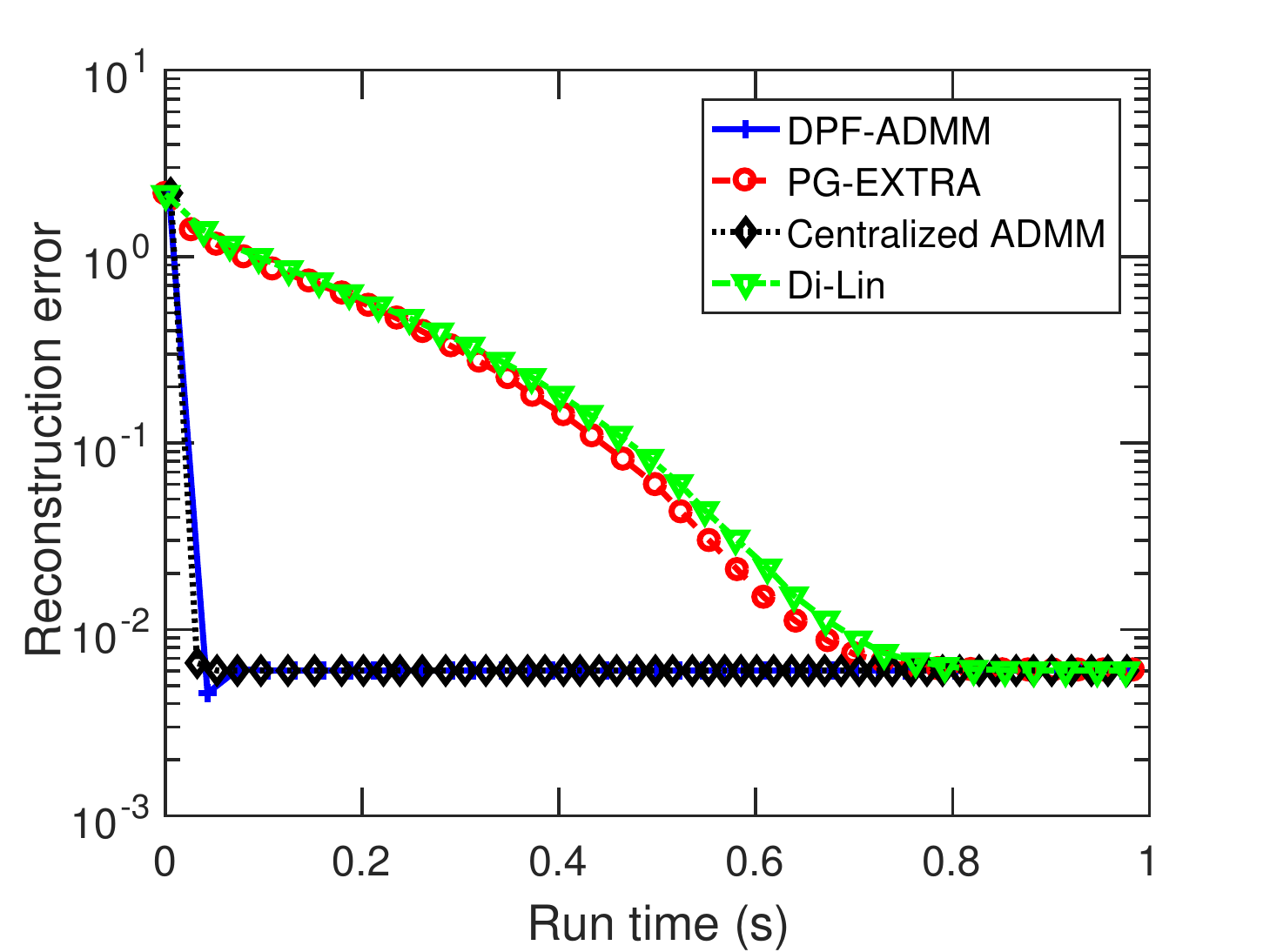}} \hfil
    \subfigure[Line graph.]{\includegraphics[width=2.3in]{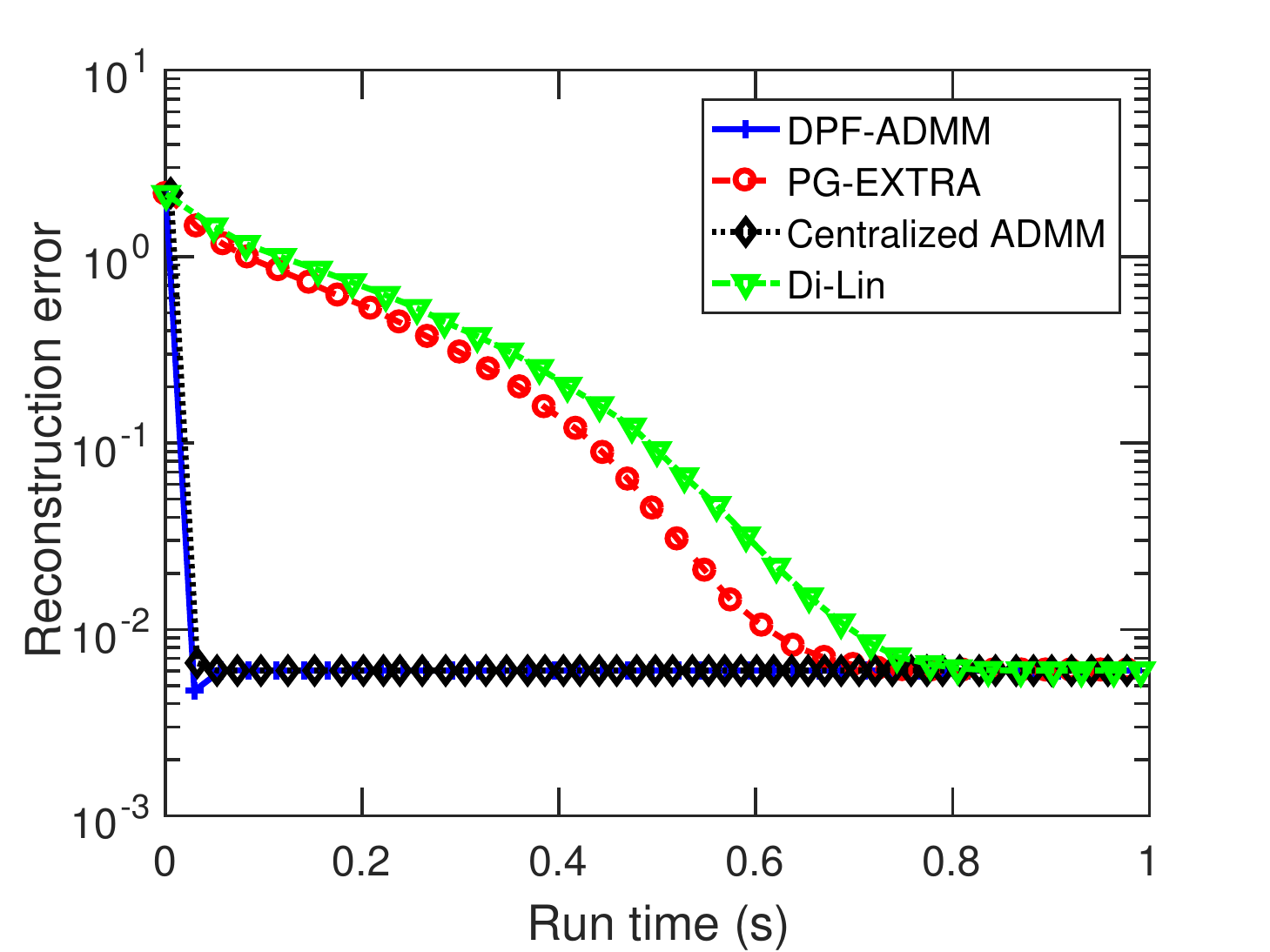}} \hfil
    \subfigure[Fully connected graph.]{\includegraphics[width=2.3in]{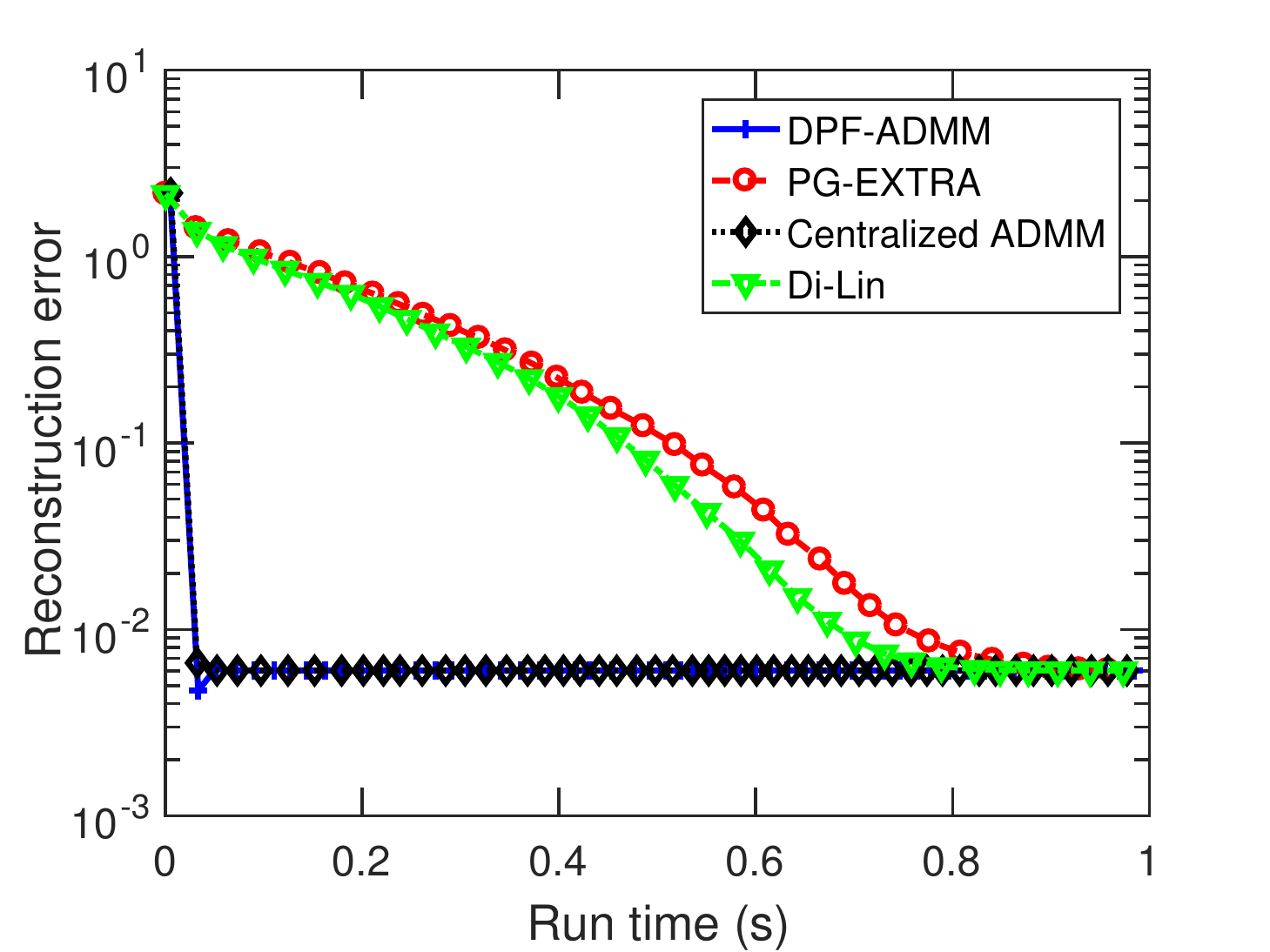}}
    \caption{
    $\ell_2$+$\ell_1$ compressed sensing: reconstruction errors vs. run time for respective algorithms.}
    \label{fig3}
\end{figure*}

\subsection{Smooth+nonsmooth: $\ell_1+\ell_2$ compressed sensing problem}
We first consider the following decentralized $\ell_1$+$\ell_2$
compressed sensing problem:
\begin{align}
\label{numer:1} \min_{\boldsymbol{x}} & \quad \sum_{i=1}^l
\frac{1}{{2\eta }}\left\| {{\boldsymbol{M}_{i}}{\boldsymbol{x}} -
\boldsymbol{b}_{i}} \right\|_2^2 +  {{{\left\| {\boldsymbol{x}}
\right\|}_1}}
\end{align}
where $\eta=1.2\times10^{-4}$, $\boldsymbol{b}_{i}
=\boldsymbol{M}_{i}\boldsymbol{x} +\boldsymbol{e}_{i}$ denotes the
measurements held by each node,
$\boldsymbol{M}_{i}\in{\mathbb{R}^{{m_i} \times n}}$ is the
sensing matrix used by each node, and $\boldsymbol{e}_{i}$ denotes
the associated noise vector. In our simulations, we set $m_i=3$,
$n=100$, the sparsity level of the sparse signal $\boldsymbol{x}$
is set equal to 5. Entries of the measurement matrix and the noise
vector are i.i.d Gaussian random variables. The variance of the
entry in the noise vector is set to $10^{-3}$.

We compare our proposed algorithm with the PG-EXTRA
\cite{ShiLing15b} and the distributed linearized ADMM (Di-Lin)
\cite{AybatWang18}. The PG-EXTRA and the Di-Lin are
state-of-the-art algorithms recently designed for smooth+nonsmooth
type of decentralized composite optimization problems. A
centralized ADMM (Section $7$ in \cite{BoydParikh11}) is also
included in our experiments for comparison. The centralized ADMM
assumes a star topology in which all nodes are individually
connected to a central point and there is no connection among
local nodes. The consensus for the centralized ADMM is achieved by
letting each local variable equal to an extra variable
$\boldsymbol{z}$ held by the central node, i.e.
$\boldsymbol{x}_i=\boldsymbol{z}$, $1\leq i\leq l$. In each
iteration, the local node computes a new $\boldsymbol{x}_i$ and
sends it to the central node, and the central node updates
$\boldsymbol{z}$ based on the received $\{\boldsymbol{x}_i\}$. The
updated $\boldsymbol{z}$ is then sent back to all local nodes. It
is widely acknowledged  \cite{ShiLing15b} that the centralized
ADMM usually provides better performance than decentralized
methods. Fig. \ref{fig2} plots the reconstruction errors vs. the
number of iterations for respective algorithms. Results are
averaged over $10^3$ independent runs, with the noise randomly
generated for each run. We see that our proposed method achieves a
much faster convergence rate than the PG-EXTRA and the Di-Lin: it
takes PG-EXTRA and Di-Lin more than $2000$ iterations to converge
(the PG-EXTRA and the Di-Lin achieve similar performance), whereas
only about $200$ iterations are needed for our algorithm to
converge. Such a performance improvement is primarily due to the
fact that our proposed method does not need to resort to any
proximal term which is known to slow down the convergence speed.
Notably, the proposed DPF-ADMM achieves performance almost the
same as that of the centralized ADMM. This result shows the
effectiveness and superiority of our proposed method. Also, since
our proposed method operates over a simplest bipartite graph with
a minimum number of edges, the amount of data to be exchanged
(among nodes) per iteration is considerably reduced compared with
the competing algorithms. Such a merit makes our proposed
algorithm particularly suitable for networks which are subject to
stringent power and communication constraints.

To more fairly evaluate the computational complexity of respective
algorithms, in Fig. \ref{fig3}, we depict the reconstruction
errors as a function of the average run time (in seconds). We see
that it takes less time for our proposed method to converge than
the PG-EXTRA and the Di-Lin.

\subsection{Nonsmooth+nonsmooth: $\ell_1+\ell_1$ compressed sensing problem}
Next we consider the $\ell_1$+$\ell_1$ compressed sensing problem:
\begin{align}
\label{numer:2} \min_{\boldsymbol{x}} & \quad
\sum_{i=1}^{l}\frac{1}{\eta}\| \boldsymbol{M}_i
\boldsymbol{x}-\boldsymbol{b}_i\|_1+\|\boldsymbol{x}\|_1
\end{align}
where $\eta$ is set to $1.2\times 10^{-3}$, and $\boldsymbol{M}_i$
is the same as before. Each noise vector $\boldsymbol{e}_i$ has
only one nonzero entry, which is a Gaussian random variable whose
variance is $10^{-2}$. Note that the proximal mapping of
$\|\boldsymbol{x}\|_1$ has a closed form solution. Nevertheless,
the proximal mapping of
$\|\boldsymbol{M}_i\boldsymbol{x}-\boldsymbol{b}\|_1$ does not
have a closed form solution. In this case, we need to transform
the problem (\ref{numer:2}) into a form that can be handled by our
proposed method. We introduce an auxiliary variable to
(\ref{numer:2}):
\begin{align}
\label{numer:3} \min_{\boldsymbol{x}, \boldsymbol{u}} &
\sum_{i=1}^{l}\frac{1}{\eta}\|\boldsymbol{u}\|_1
+\phi(\boldsymbol{u}-
\boldsymbol{M}_i\boldsymbol{x}-\boldsymbol{b}_i)+
\|\boldsymbol{x}\|_1
\end{align}
where $\phi(\boldsymbol{x})$ is an indicator function, whose value
is $+\infty$ if $\boldsymbol{x}\neq\boldsymbol{0}$, and
$\boldsymbol{0}$ otherwise. Let
$f_i(\boldsymbol{x};\boldsymbol{u})=\frac{1}{\eta}\|
\boldsymbol{u}\|_1+\phi(\boldsymbol{u}-\boldsymbol{M}_i\boldsymbol{x}
-\boldsymbol{b}_i)$ and
$g_i(\boldsymbol{x})=\|\boldsymbol{x}\|_1$, then the problem
(\ref{numer:3}) can be written as
\begin{align}
\min_{\boldsymbol{x},\boldsymbol{u}} & \quad
\sum_{i=1}^{l}f_i(\boldsymbol{x};\boldsymbol{u})+g_i(\boldsymbol{x})
\label{numer:4}
\end{align}
which is similar to (\ref{opt1}). Following the derivation in
Section \ref{sec:PF-ADMM-composite}, the decentralized formulation
of the problem (\ref{numer:4}) can be described as
\begin{align}
\mathop {\min }\limits_{{\tilde{\boldsymbol{x}}}_f,
{\tilde{\boldsymbol{u}}}_f \ {\tilde{\boldsymbol{y}}}_g} & \quad
\tilde{f}({\tilde{\boldsymbol{x}}}_f;{\tilde{\boldsymbol{u}}}_f)+
\tilde{g}({\tilde{\boldsymbol{y}}}_g) \nonumber\\
\text{s.t.} & \quad \boldsymbol{C}_H{\tilde{\boldsymbol{x}}}_f
+\boldsymbol{C}_T{\tilde{\boldsymbol{y}}}_g=\boldsymbol{0}
\label{numer:5}
\end{align}
where all the notations except ${\tilde{\boldsymbol{u}}}_f^T=
[{\tilde{\boldsymbol{u}}}_H^T \ {\tilde{\boldsymbol{u}}}_T^T]^T$
is the same as those in Section \ref{sec:PF-ADMM-composite}. Note
that the variable ${\tilde{\boldsymbol{u}}}_f$ is kept locally by
each node and does not participate in the information exchange. It
is difficult to handle the indicator function and the $\ell_1$
norm in $\tilde{f}$ simultaneously, as a compromise, we choose to
rewrite the indicator function as an equality constraint.
Attaching an additional Lagrangian multiplier to the this equality
constraint, now we can apply our proposed DPF-ADMM to solve
(\ref{numer:5}).

Most existing decentralized algorithms are designed exclusively
for solving smooth+nonsmooth problems, including the PG-EXTRA.
Thus, we choose to compare our proposed algorithm with the DGD
\cite{NedicOzdaglar09} and the centralized ADMM. For the competing
algorithms, the parameters are tuned to achieve the best
performance. Fig. \ref{fig4} plots reconstruction errors vs. the
number of iterations for respective algorithms. Results are
averaged over $10^{4}$ runs, with the noise randomly generated for
each run. We see that the our proposed DPF-ADMM presents a
significant performance advantage over the DGD in terms of both
reconstruction accuracy and convergence speed. Due to the simple
structure, the computation time (per iteration) of the DGD is
about half of that of our proposed algorithm. Nevertheless, the
DPF-ADMM still exhibits an overwhelming advantage over the DGD in
terms of the computational complexity, see Fig. \ref{fig5}.

\begin{figure*}[t]
    \centering
    \subfigure[Random graph.]{\includegraphics[width=2.3in]{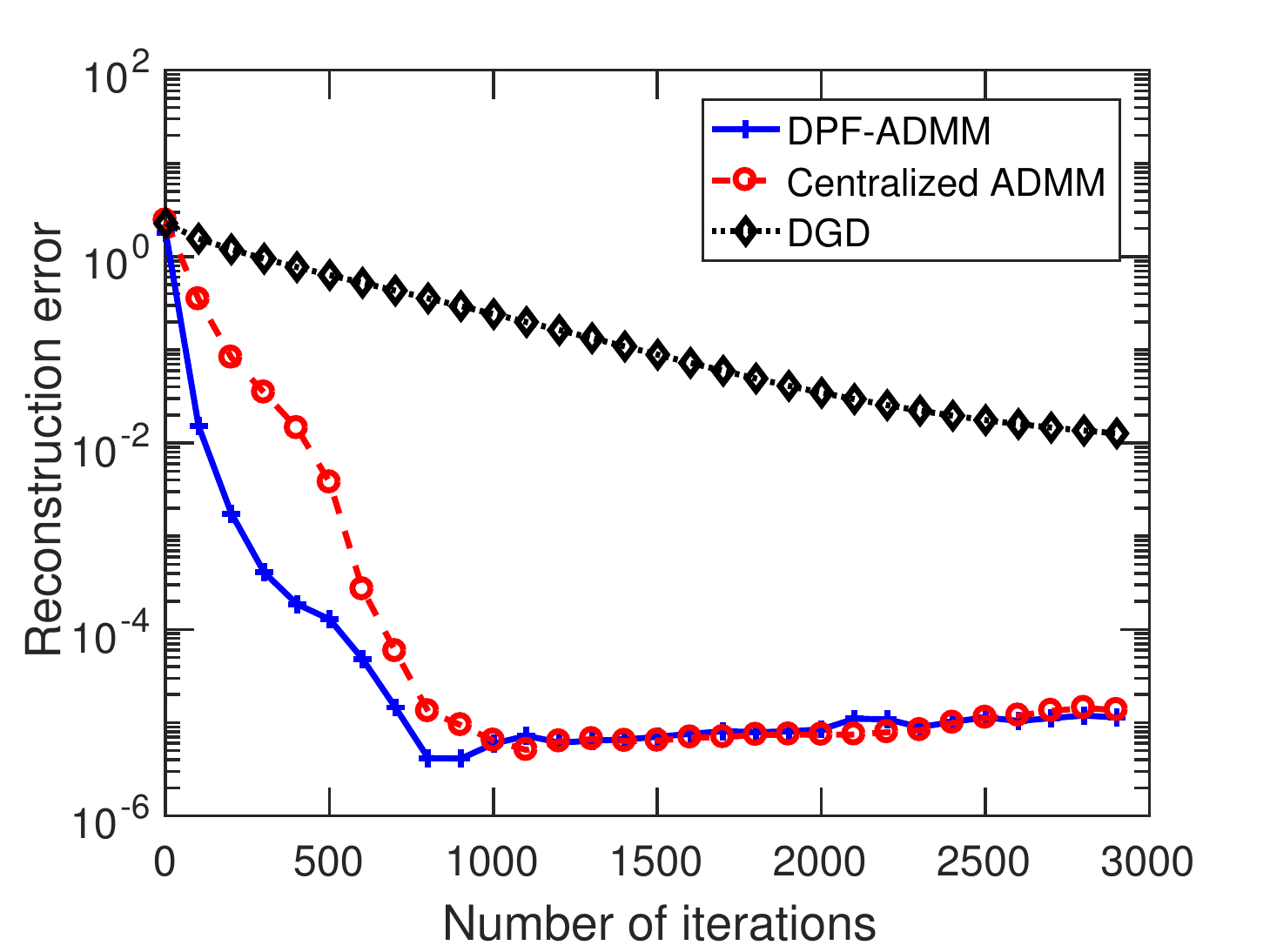}} \hfil
    \subfigure[Line graph.]{\includegraphics[width=2.3in]{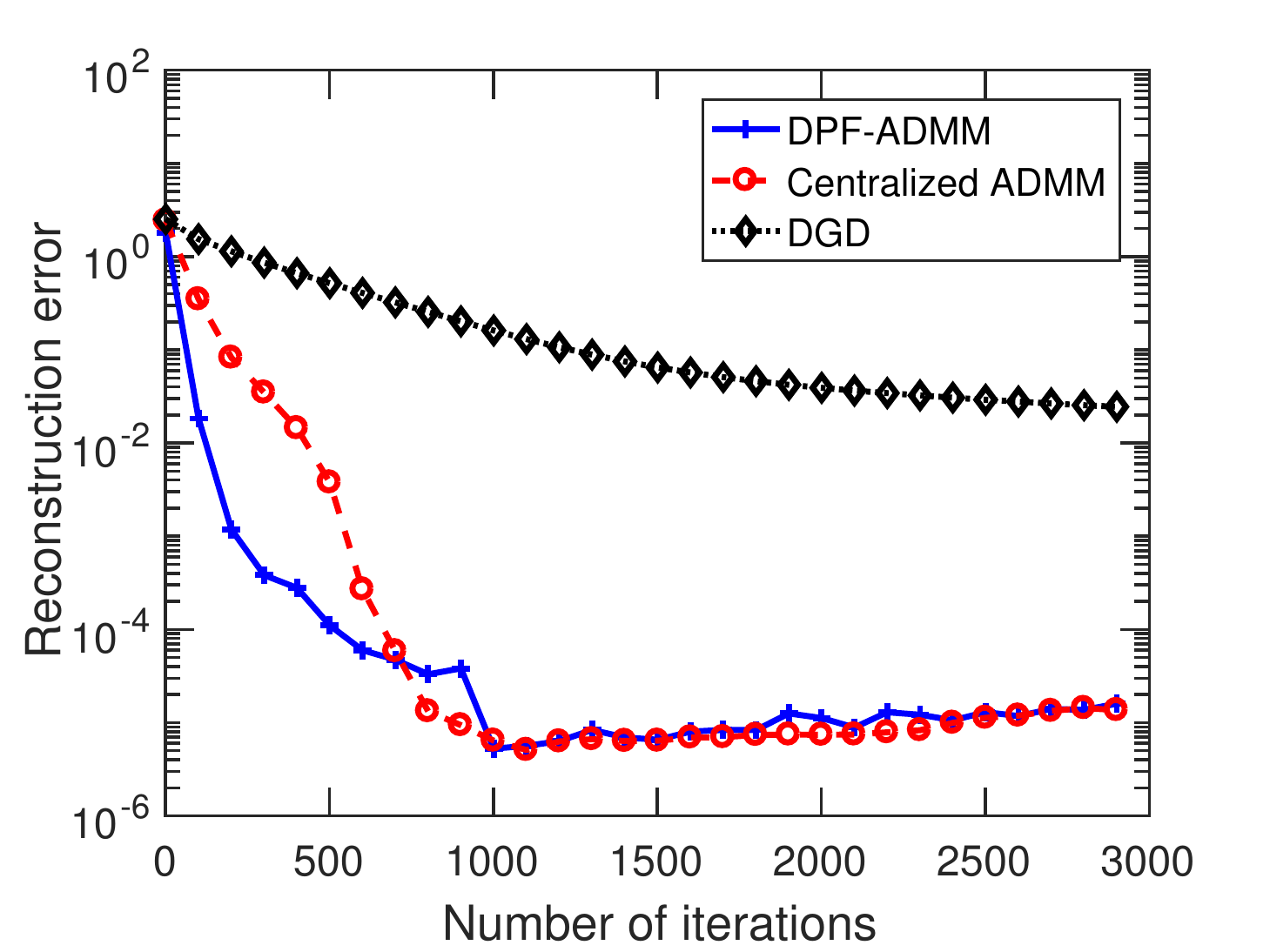}} \hfil
    \subfigure[Fully connected graph.]{\includegraphics[width=2.3in]{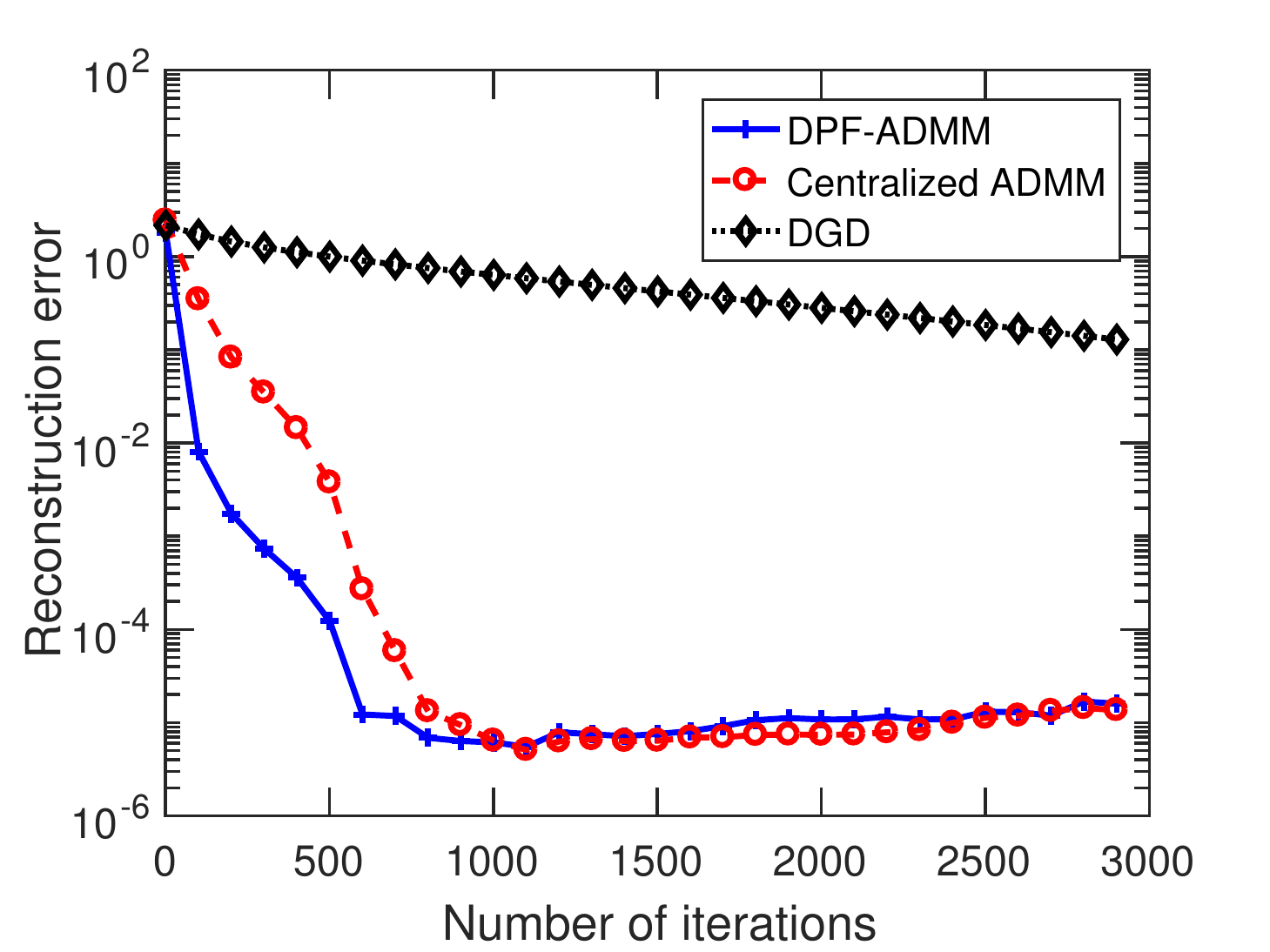}}
    \caption{
    $\ell_1+\ell_1$ compressed sensing: reconstruction errors vs. the number of iterations for respective algorithms.}
    \label{fig4}
\end{figure*}

\begin{figure*}[t]
    \centering
    \subfigure[Random graph.]{\includegraphics[width=2.3in]{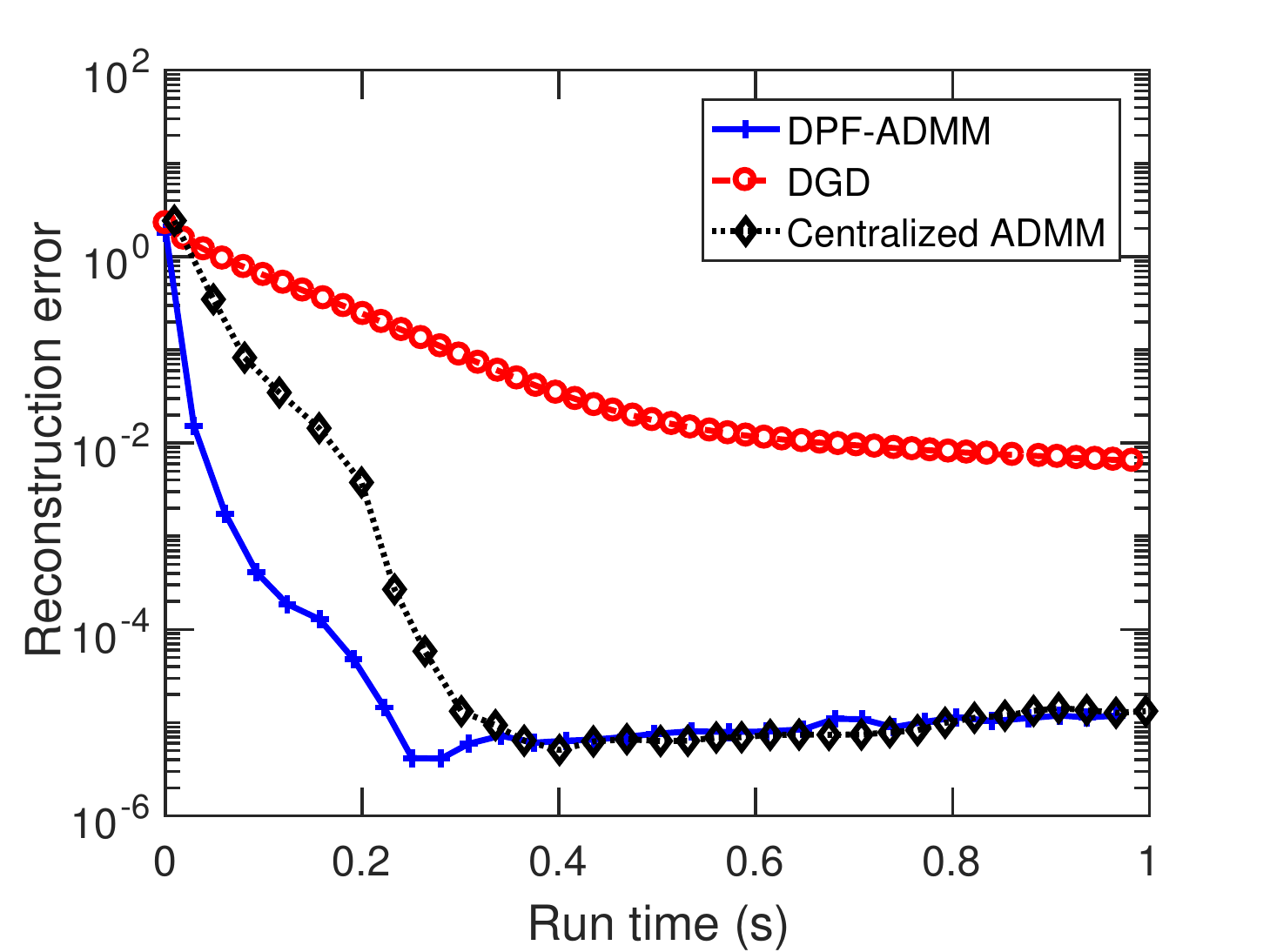}} \hfil
    \subfigure[Line graph.]{\includegraphics[width=2.3in]{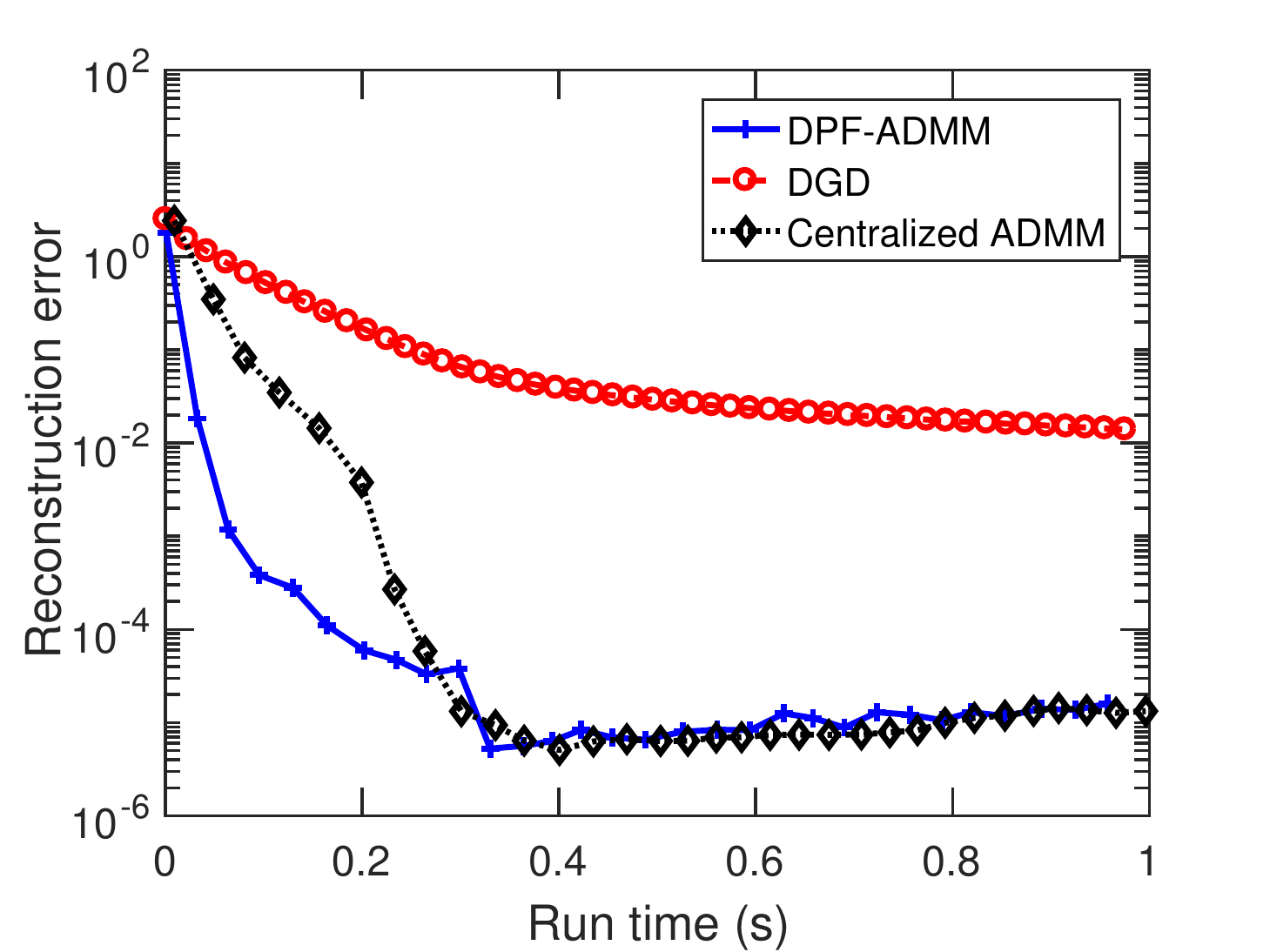}} \hfil
    \subfigure[Fully connected graph.]{\includegraphics[width=2.3in]{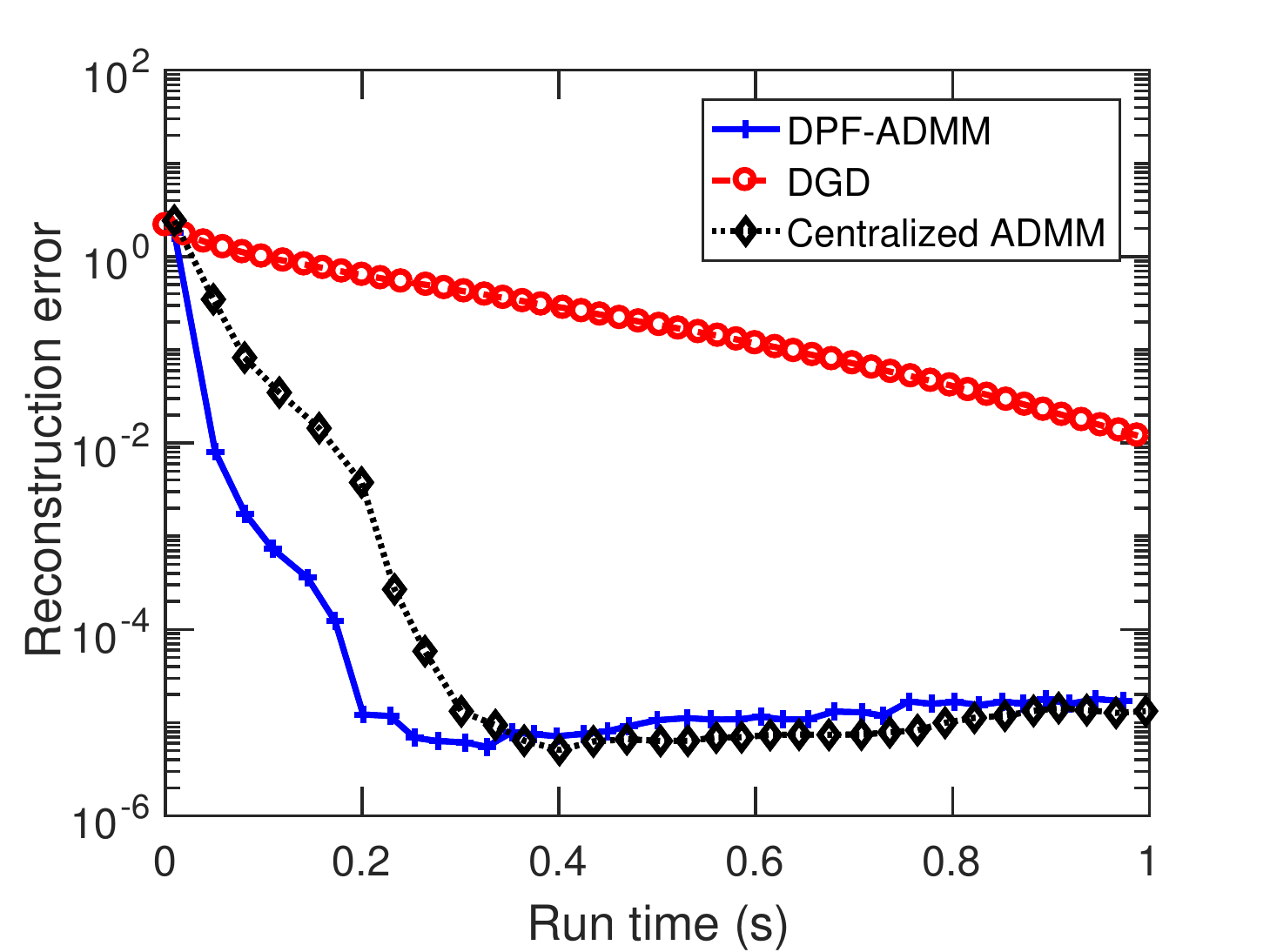}}
    \caption{
    $\ell_1+\ell_1$ compressed sensing: reconstruction errors vs. average run time for respective algorithms.}
    \label{fig5}
\end{figure*}

\section{Conclusion} \label{sec:conclusion}
In this paper, we proposed a decentralized proximal-free ADMM for
solving decentralized composite optimization problems. The
proposed algorithm was developed based on a simplest bipartite
graph, which is defined as a bipartite graph that has a minimum
number of edges to keep the graph connected. We showed that the
simplest bipartite graph has some interesting properties that can
be utilized to develop a decentralized algorithm without involving
any proximal terms. The proposed algorithm exhibits a much faster
convergence speed than state-of-the-art decentralized algorithms.
Notably, it achieves performance similar to that of the
centralized ADMM, which is known to provide the best achievable
performance for all decentralized methods. In addition, the
proposed algorithm entails a minimal communication cost because it
runs on a simplest bipartite graph, which has the minimum number
of edges that keeps the graph connected.

\useRomanappendicesfalse
\appendices

\section{Proof of Lemma \ref{lemma2}} \label{appA}
To prove that the matrix $\boldsymbol{X}$ has full row rank, we
only need to show that the dimension of the null space of
$\boldsymbol{X}^T$ is equal to one because $\boldsymbol{X}$ is an
$(l-1)\times l$ matrix. From Lemma \ref{lemma1}, we know that each
row of $\boldsymbol{X}$ has only two nonzero elements that are of
equal magnitude but opposite signs. Therefore it is clear that an
all-one column vector $\boldsymbol{1}$ lies in the null space of
$\boldsymbol{X}^T$, i.e.
$\boldsymbol{X}\boldsymbol{1}=\boldsymbol{0}$. On the other hand,
in the following, we show that for any vector
$\boldsymbol{z}\notin \text{span}\{ \boldsymbol{1}\}$, we have
$\boldsymbol{X}\boldsymbol{z}\neq\boldsymbol{0}$.

For any vector $\boldsymbol{z}\notin \text{span}\{
\boldsymbol{1}\}$, we can always find two entries of
$\boldsymbol{z}$, say, $z_i$ and $z_j$, that are nonidentical,
i.e. $z_i\neq z_j$. Suppose there is an edge between node $i$ and
node $j$ in the simplest bipartite graph, and let
$\boldsymbol{x}^T$ denote the row in $\boldsymbol{X}$
corresponding to this edge. According to the results in Lemma
\ref{lemma1}, we know that the $i$th and $j$th entries of
$\boldsymbol{x}^T$ have a magnitude of $1$ but opposite signs,
whereas other entries equal to zero. In this case, we have
$\boldsymbol{x}^T\boldsymbol{z}\neq 0$, which implies
$\boldsymbol{X}\boldsymbol{z}\neq\boldsymbol{0}$. Now consider the
case where there is no direct link (i.e. edge) between node $i$
and node $j$. We first assume
$\boldsymbol{X}\boldsymbol{z}=\boldsymbol{0}$ and proceed our
proof by contradiction. Since the simplest bipartite graph is a
connected graph, there always exists a route to connect these two
nodes. Suppose there are $s$ intermediate nodes, say node $k_1$,
$\ldots$, $k_s$, between node $i$ and node $j$. Since there is an
edge between node $i$ and node $k_1$, its corresponding row in
$\boldsymbol{X}$, denoted as $\boldsymbol{x}_{t_1}^T$, has two
nonzero entries (i.e. the $i$th entry and the $k_1$th entry) of
equal magnitude but opposite signs. From
$\boldsymbol{X}\boldsymbol{z}=\boldsymbol{0}$, we can arrive at
$z_i=z_{k_1}$. Following a similar derivation, we can establish a
sequence of equations: $z_{k_1}=z_{k_2},\ldots,z_{k_s}=z_j$.
Combining these equations, we eventually reach the conclusion that
$z_i=z_j$, which contradicts our original assumption $z_i\neq
z_j$. Therefore $\boldsymbol{X}\boldsymbol{z}=\boldsymbol{0}$
cannot be true. In other words, for any vector
$\boldsymbol{z}\notin \text{span}\{ \boldsymbol{1}\}$, we have
$\boldsymbol{X}\boldsymbol{z}\neq\boldsymbol{0}$, which implies
that the dimension of the null space of $\boldsymbol{X}^T$ is
equal to one. Hence $\boldsymbol{X}$ has full row rank.

\bibliography{main}
\bibliographystyle{IEEEtran}
\end{document}